\documentclass[11pt,a4paper]{article}
\usepackage{amsmath, amsfonts, amsthm, amssymb, graphicx, xcolor, dsfont, enumerate, dcolumn,subcaption}
\usepackage[numbers,square, comma, compress]{natbib}
\usepackage{threeparttable}

\usepackage{hyperref}

\usepackage{tikz}

\newtheoremstyle{mystyle}
  {}
  {}
  {\itshape}
  {}
  {\bfseries}
  {.}
  { }
  {}

\theoremstyle{mystyle}

\newtheorem{lemma}{Lemma}

\newtheorem{theorem}{Theorem}
\newtheorem{proposition}{Proposition}
\newtheorem{corollary}{Corollary}
\newtheorem{conjecture}{Conjecture}
\newtheorem{observation}{Observation}
\newtheorem{ques}{Question}

\newcommand{\h}{{\frak h}}

\begin{document}

\title
{Few hamiltonian cycles in graphs\\ with one or two vertex degrees}
\author{{\sc Jan GOEDGEBEUR\footnote{Department of Computer Science, KU Leuven Campus Kulak-Kortrijk, 8500 Kortrijk, Belgium}\;\footnote{Department of Applied Mathematics, Computer Science and Statistics, Ghent University, 9000 Ghent, Belgium}\;,
Jorik JOOKEN\footnotemark[1]\;,
On-Hei Solomon LO\footnote{Faculty of Environment and Information Sciences, Yokohama National University, 79-2 Tokiwadai, Hodogaya-ku, Yokohama 240-8501, Japan}\;, }\\[1mm]
Ben SEAMONE\footnote{Mathematics Department, Dawson College, Montreal, QC, Canada}\;\footnote{D\'epartement d'informatique et de recherche op\'erationnelle, Universit\'e de Montr\'eal, Montreal, QC, Canada}\;\,, and {\sc Carol T. ZAMFIRESCU\footnotemark[2]\;\footnote{Department of Mathematics, Babe\c{s}-Bolyai University, Cluj-Napoca, Roumania}}\;\footnote{E-mail addresses: jan.goedgebeur@kuleuven.be; jorik.jooken@kuleuven.be;  ohsolomon.lo@gmail.com; ben.seamone@gmail.com; czamfirescu@gmail.com}}

\date{}

\maketitle

\begin{abstract}
\noindent We fully disprove a conjecture of Haythorpe on the minimum number of hamiltonian cycles in regular hamiltonian graphs, thereby extending a result of Zamfirescu, as well as correct and complement Haythorpe's computational enumerative results from [\emph{Experim.\ Math.} \textbf{27} (2018) 426--430]. Thereafter, we use the Lov\'asz Local Lemma to extend Thomassen's independent dominating set method. Regarding the limitations of this method, we answer a question of Haxell, Seamone, and Verstraete, and settle the first open case of a problem of Thomassen. Motivated by an observation of Aldred and Thomassen, we prove that for every $\kappa \in \{ 2, 3 \}$ and any positive integer $k$, there are infinitely many non-regular graphs of connectivity~$\kappa$ containing exactly one hamiltonian cycle and in which every vertex has degree $3$ or $2k$.

\bigskip\noindent \textbf{Keywords:} Hamiltonian cycle; regular graph; Lov\'asz Local Lemma

\bigskip\noindent \textbf{MSC 2020:} 05C45, 05C85 
\end{abstract}


\section{Introduction}
\label{sect:intro}

Motivated by Sheehan's famous conjecture stating that every hamiltonian 4-regular graph contains at least two hamiltonian cycles~\cite{Sh75}, Fleischner's surprising result that there exist graphs in which every vertex has degree 4 or 14 and which contain exactly one hamiltonian cycle~\cite{Fl14}, as well as recent work of Haythorpe~\cite{Ha18}, in this paper we investigate bounds for the minimum number of hamiltonian cycles occurring in hamiltonian graphs in which the set of distinct vertex degrees contains at most two elements. For further results treating Sheehan's conjecture and its variations we point to the recent articles~\cite{GKN19,GMZ20,Za22} and references therein.

We say that a graph is \textit{$(k,\ell)$-regular} if all of its vertices are of degree $k$ or $\ell$ and there is at least one vertex of degree $k$ and at least one vertex of degree $\ell$; $k = \ell$ is allowed, in which case we simply write \emph{$k$-regular}. 
We will denote by $h_n(k,\ell)$ the minimum number of hamiltonian cycles of any hamiltonian $(k,\ell)$-regular graph of order $n$, and put
$$h_n(k) := h_n(k,k), \quad h(k, \ell) := \min_{n}h_n(k, \ell), \quad {\rm and} \quad h(k) := \min_{n}h_n(k).$$
In case no hamiltonian $(k,\ell)$-regular graph of order $n$ exists, we write $h_n(k,l):=\infty$. Aside from trivial lower bounds, related to the important quartic case we know that $h(4) \le 12$, $\limsup_{n \to \infty} h_n(4) \le 144$, and $h(3,4) = h(4,14) = 1$, see~\cite{GMZ20,Za22,ES80,Fl14}, respectively. With the same idea as in the proof of Theorem~1 in~\cite{Za22}, one can show that for every integer $\ell \ge 4$ there exists a constant $c_\ell$, depending only on $\ell$, such that for infinitely many $n$ we have $h_n(4,\ell) \le c_\ell$. However, only for few $\ell$ the exact value of $c_\ell$ is known.

We will here be mostly concerned with \emph{upper} bounds for the above quantities, but the reader might wonder what can be said about \emph{lower} bounds. Unfortunately, not much. By extending a technique of Thomassen~\cite{Th98} it was proven by Haxell, Seamone, and Verstraete~\cite{HSV07} that $h(k) \ge 2$ for all $k \ge 23$. If both $k$ and $\ell$ are odd, then it follows from Thomason's ``lollipop'' technique~\cite{Th78} that $h(k,\ell) \ge 3$. On the other hand, expanding on what has been said above, Entringer and Swart~\cite{ES80} proved $h(3,4) = 1$ (further examples were given by Aldred and Thomassen in~\cite{HA99}, Royle~\cite{Ro17}, and in~\cite{GMZ20}) and Fleischner~\cite{Fl14}, as mentioned above, showed $h(4,14) = 1$.

Whenever a figure in this article shows, for a particular graph $G$, some of $G$'s edges displayed thicker than others, then the set of thin edges belongs to no hamiltonian cycle of $G$, and every thick edge belongs to at least one hamiltonian cycle of $G$---this facilitates checking certain arguments.

A number of results in this paper were obtained using computer-aided methods. 
In total these computations amounted to about 35 CPU years. 
To ensure the correctness of the results, we performed independent verifications in which we solve the same problem using different algorithms; we also present a human-readable proof where possible. We made the source code of our implementations and the certificates that can be used to independently verify our claims publicly available on GitHub~\cite{code}.

This article is organised as follows. In Section~\ref{sect:haythorpe_conjecture} we discuss a recent conjecture of Haythorpe~\cite{Ha18} revolving around the minimum number of hamiltonian cycles in hamiltonian $k$-regular graphs, where $k \ge 5$. In~\cite{Za22} it was proven that this conjecture does not hold for $k \in \{ 5, 6, 7 \}$. We here extend this result and show that it holds for \emph{no} $k$. In Section~\ref{sect:haythorpe_table} we correct certain inaccuracies occurring in Haythorpe's article~\cite{Ha18}---these revolve around computational results on the counting of hamiltonian cycles in regular graphs. In Section~\ref{sect:rigd} we first use probabilistic arguments involving the Lov\'asz Local Lemma to extend Thomassen's independent dominating set method. Thereafter, combining a gluing argument with computational methods, it is shown that for $k \in \{ 5, 6\}$ there exist infinitely many $k$-regular graphs with a hamiltonian cycle $\h$ containing no $\h$-independent dominating set; this answers a question of Haxell, Seamone, and Verstraete~\cite{HSV07} and settles the first open case of a problem of Thomassen~\cite{Th98}. The paper concludes with Section~\ref{sect:kl_regular}, in which we extend an observation of Aldred and Thomassen, proving that for every $\kappa \in \{ 2, 3 \}$ and any positive integer $k$, there are infinitely many non-regular graphs of connectivity~$\kappa$ containing exactly one hamiltonian cycle and in which every vertex has degree $3$ or $2k$. It is also shown that the smallest $(3,4)$-regular graph that contains exactly one hamiltonian cycle has order~18, and that all $(3, 2k + 1)$-regular graphs with exactly three hamiltonian cycles and at most 32~vertices are cubic.

\section{Counterexamples to Haythorpe's conjecture}
\label{sect:haythorpe_conjecture}

In~\cite{Ha18} Haythorpe published the following conjecture.
\begin{conjecture}[Conjecture 3.1 in~\cite{Ha18}]\label{conj:haythorpe}
For $d \ge 5$ and $n \ge d+3$, all hamiltonian $d$-regular graphs of order $n$ have at least $f(n,d) := (d-1)^2 [(d-2)!]^{\frac{n}{d+1}}$ hamiltonian cycles.
\end{conjecture}

\subsection{Counterexamples for all values of \boldmath$d$}

In~\cite{Za22} Conjecture~\ref{conj:haythorpe} was shown to not hold for $d \in \{ 5, 6, 7 \}$. By extending the approach from~\cite{Za22}, we show that Haythorpe's conjecture holds for \emph{no} integer $d \ge 5$.

\begin{theorem}
\label{haythorpeCounterExampleTheorem}

For arbitrary integers $k \ge 0$ and $d \ge 5$, we have
$$h_n(d) \le 2[(d-1)!]^{d-2}[(d-2)!]^k < f(n,d),$$
where $n := d^2+d-4+(d+1)k$.
\end{theorem}
\begin{proof}

The proof of the first inequality is structurally identical to the proof of Theorem~3 from~\cite{Za22}, so we will be succinct. Fig.~\ref{constructionFigure} shows a $d$-regular graph on $d^2+d-4$ vertices---see the figure's caption for an exact description of the graph's structure---with exactly $2((d-1)!)^{d-2}$ hamiltonian cycles. 
In Fig.~\ref{constructionFigure}, the graph induced by $v, w$, and the vertices marked by white disks has order~$d+1$. 
It has exactly two hamiltonian $vw$-paths. By applying Lemma~2 from~\cite{Za22}, this first part of the proof is completed.

\begin{figure}[h!]
\begin{center}
  \includegraphics[height=70mm]{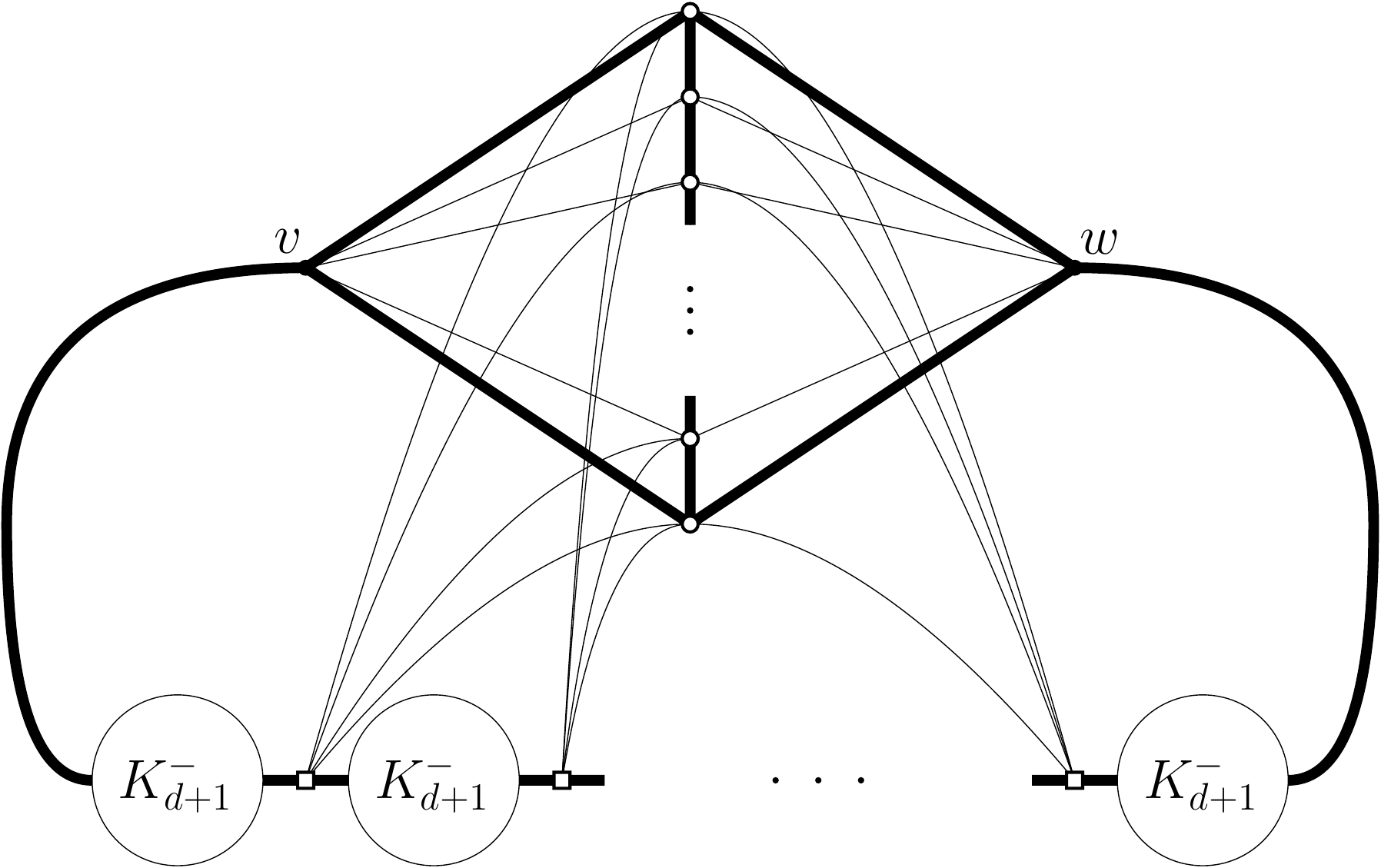}
  \caption{A $d$-regular graph. $K_{d+1}^-$ stands for a complete graph on $d+1$ vertices minus the horizontal edge, and there are $d - 2$ copies of these graphs present. There are $d-3$ white squares and $d-1$ white disks. No two white squares are connected; the graph induced by the set of all white disks is a path; and white squares and white disks are connected in any way that produces a $d$-regular graph---this is possible as the number of edges leaving white squares, $(d-3)(d-2)$, is equal to the number of edges leaving white disks, $2(d-3) + (d-3)(d-4)$.}\label{constructionFigure}
\end{center}
  \end{figure}

We now prove the second inequality, i.e.\ show that
$$2[(d-1)!]^{d-2}[(d-2)!]^k<(d-1)^2 [(d-2)!]^{\frac{d^2+d-4+(d+1)k}{d+1}}.$$
Rewrite the left-hand side as
\begin{align*}
2(d-1)^{d-2}[(d-2)!]^{d-2}[(d-2)!]^k
\end{align*}
and the right-hand side as
\begin{align*}
(d-1)^2 [(d-2)!]^{d+k-\frac{4}{d+1}}=(d-1)^2 [(d-2)!]^{2-\frac{4}{d+1}} [(d-2)!]^{d-2}[(d-2)!]^{k}.
\end{align*}
By cancelling out common factors, this leads to

$$2[(d-1)!]^{d-2}[(d-2)!]^k <(d-1)^2 [(d-2)!]^{\frac{d^2+d-4+(d+1)k}{d+1}},$$
i.e.
$$2(d-1)^{d-4} <[(d-2)!]^{2-\frac{4}{d+1}}.$$

For $5 \le d \le 59$, we verified with a computer that the inequality holds. For $d \ge 60$, the inequality can be proven via Stirling's formula, and by using the fact that a polynomial of degree four with variable $d$ is larger than a polynomial of degree three for a sufficiently large value of $d$ (in our case, $d \ge 60$):
\begin{equation*}
\begin{split}
[(d-2)!]^{2-\frac{4}{d+1}} & > [(d-2)!]^{\frac{3}{2}} >  \left[\sqrt{2 \pi (d-2)} \left(\frac{d-2}{e}\right)^{d-2}\right]^{\frac{3}{2}} > 2 \left[\left(\frac{d-2}{e}\right)^{4}\right]^{\frac{3}{8}d-\frac{3}{4}}\\
& > 2 \left[\left(d-1\right)^{3}\right]^{\frac{3}{8}d-\frac{3}{4}}>2(d-1)^{d-4}.
\end{split}
\end{equation*}
\end{proof}

\subsection{Improved upper bounds for the \boldmath$6$- and \boldmath$7$-regular case}

The following result was proven in~\cite{Za22}. As already mentioned, it shows that Conjecture~\ref{conj:haythorpe} does not hold for $d \in \{5, 6, 7 \}$.

\begin{theorem}[Theorem 3 from~\cite{Za22}]
\label{ZamfirescuTheorem}
For every non-negative integer $k$ there exists (i) a $5$-regular graph on $26 + 6k$ vertices with exactly $2^{k+10} \cdot 3^{k+3}$ hamiltonian cycles, (ii) a $6$-regular graph on $39 + 7k$ vertices
with exactly $2^{3(k+4)} \cdot 3^{k+4} \cdot 5^{5}$ hamiltonian cycles, and (iii) a $7$-regular graph on $54+8k$ vertices with exactly $2^{3(k+8)} \cdot 3^{k+10} \cdot 5^{k+5}$ hamiltonian cycles.
\end{theorem}

If we set $d=5$, $d=6$, and $d=7$ in Theorem~\ref{haythorpeCounterExampleTheorem}, we get:

\begin{corollary}
\label{567Corollary}
For every non-negative integer $k$ there exists (i) a $5$-regular graph on $26 + 6k$ vertices with exactly $2^{k+10} \cdot 3^{k+3}$ hamiltonian cycles, (ii) a $6$-regular graph on $38 + 7k$ vertices
with exactly $2^{3(k+4)+1} \cdot 3^{k+4} \cdot 5^{4}$ hamiltonian cycles, and (iii) a $7$-regular graph on $52+8k$ vertices with exactly $2^{3(k+7)} \cdot 3^{k+10} \cdot 5^{k+5}$ hamiltonian cycles.
\end{corollary}

Corollary~\ref{567Corollary} results in the same number of hamiltonian cycles (and the same order) for $d=5$ as Theorem~\ref{ZamfirescuTheorem}. However, for $d=6$ and $d=7$ the orders are not the same ($38+7k$ for Corollary~\ref{567Corollary} versus $39+7k$ for Theorem~\ref{ZamfirescuTheorem} using $d=6$, and $52+8k$ for Corollary~\ref{567Corollary} versus $54+8k$ for Theorem~\ref{ZamfirescuTheorem} using $d=7$). We now present a variation of Theorem~\ref{haythorpeCounterExampleTheorem} such that the number of vertices is the same for $d=6$ and $d=7$, but the number of hamiltonian cycles is smaller:

\begin{theorem}
\label{variationTheorem}
For every non-negative integer $k$ there exists (i) a $6$-regular graph on $39 + 7k$ vertices
with exactly $2^{3(k+4)+1} \cdot 3^{k+4} \cdot 5^{4}$ hamiltonian cycles, and (ii) a $7$-regular graph on $54+8k$ vertices with exactly $2^{3(k+7)} \cdot 3^{k+10} \cdot 5^{k+5}$ hamiltonian cycles.
\end{theorem}
\begin{proof}
The proof is very similar to the proof of Theorem~\ref{haythorpeCounterExampleTheorem}, only a certain subgraph has been replaced. The graphs can be found in Fig.~\ref{sixRegularVariation} and Fig.~\ref{sevenRegularVariation} (in the Appendix) for $d=6$ and $d=7$, respectively.
\end{proof}

\section{Correcting and extending Haythorpe's table}
\label{sect:haythorpe_table}

Haythorpe published several values of $h_n(k)$ and upper bounds for $h_n(k)$ in Table 1 of~\cite{Ha18}, for integers $k$ and $n$ satisfying $4 \leq k \leq 7$ and $5 \leq n \leq 18$. According to~\cite{Ha18}, the values of $h_n(k)$ were ``computed by first using GENREG~\cite{Me99} to construct all $k$-regular graphs on $n$ vertices for various values of $n$ and $k$, and then using the hamiltonian cycle enumeration algorithm by Chalaturnyk~\cite{Ch08} to count the number of hamiltonian cycles in each.'' 

The upper bounds for $h_n(k)$ from~\cite{Ha18} are equal to $h_n^2(k)$, which we define as the minimum number of hamiltonian cycles amongst all hamiltonian $k$-regular graphs on $n$ vertices with vertex connectivity~2. These values were computed in~\cite{Ha18} for the cases where there were more than 100 million connected $k$-regular graphs on $n$ vertices, motivated by the observation that the minimal examples for the values of $h_n(k)$ that were calculated exhaustively always had connectivity~2 if any such graph existed for that choice of $n$ and $k$ (i.e.\ (i) $n \geq 2k+2$ and (ii)~$k$~even or $k$ odd and $n$ even). However, as we will show in this section, the reported values of $h_n^2(k)$ in~\cite{Ha18} are inaccurate, i.e.\ there exist regular hamiltonian graphs with connectivity~2 for which the number of hamiltonian cycles is smaller than the reported values.

We repeated Haythorpe's experiment, using two independent algorithms for counting the number of hamiltonian cycles. To enumerate all hamiltonian cycles, the first algorithm starts from an initially empty path that only contains the vertex with the lowest degree---ties are broken arbitrarily---and it will recursively extend the path in all possible ways by visiting an unvisited vertex. If there are no more unvisited vertices left that can be reached, the algorithm checks whether all vertices have been visited and whether there is an edge between the end vertices of the path (to obtain a hamiltonian cycle). In order to speed up this approach, the algorithm avoids double counting cycles by breaking symmetry and the algorithm prunes paths that cannot lead to a hamiltonian cycle, based on the degrees of the unvisited vertices and the end vertices of the path. This approach is based on the algorithm from~\cite{GRWZ22}.

The second algorithm uses the dynamic programming algorithm due to Held and Karp~\cite{HK62} for solving sequencing problems---the hamiltonian cycle problem can be seen as a special case of such a sequencing problem by applying Lemma~1 of~\cite{JLD20}. For both algorithms, we also implemented a version that counts hamiltonian paths instead of cycles, since we will use this later.

We first computed the values of $h_n(k)$ by using GENREG~\citep{Me99} to generate all connected $k$-regular graphs on $n$ vertices. Thereafter, we used the two independent algorithms for computing the number of hamiltonian cycles. The values of $h_n(k)$ obtained through this experiment match the values that were reported by Haythorpe. For computing the values of $h_n^2(k)$, we used a different approach starting from \textit{nearly $k$-regular graphs}, i.e.\ graphs $G$ in which every vertex has degree $k$, except for two vertices which have degree less than $k$. We will refer to these two vertices as the \textit{terminals} of $G$. The following observation, whose proof is straightforward, shows the relationship between hamiltonian $k$-regular graphs of connectivity 2 and nearly $k$-regular graphs. 

\begin{observation}
\label{nearlyRegularObservation}
Let $G_1=(V_1,E_1)$ be a hamiltonian $k$-regular graph with $p_1$ hamiltonian cycles and let $\{u, v\}$ be a $2$-cut of $G_1$. The graph $G_1-u-v$ consists of precisely two connected components $G_2=(V_2,E_2)$ and $G_3=(V_3,E_3)$ for which $|V_2| \geq k-1$ and $|V_3| \geq k-1$. Let $p_2$ and $p_3$ be the number of hamiltonian paths between the terminals of the nearly $k$-regular graphs $G_1[V_2 \cup \{u, v\}]$ and $G_1[V_3 \cup \{u, v\}]$, respectively. We have $p_1 = p_2 p_3$, since $G_1$ can be reconstructed by identifying the corresponding terminals of $G_1[V_2 \cup \{u, v\}]$ and $G_1[V_3 \cup \{u, v\}]$.
\end{observation}

Observation~\ref{nearlyRegularObservation} allows us to compute $h_n^2(k)$ by first generating all connected nearly $k$-regular graphs on at most $n-k+1$ vertices, which we did using the generator GENG~\citep{MP14} and a filter that checks whether a graph is nearly $k$-regular. Secondly, we computed the number of hamiltonian paths between the two terminals of these nearly $k$-regular graphs and combined the ones with a minimal number of hamiltonian paths (by identifying the terminals) to obtain complete lists of the $k$-regular graphs on $n$ vertices with $h_n^2(k)$ hamiltonian cycles. Nearly $k$-regular graphs are interesting from a computational point of view, because they allow to generate all hamiltonian $k$-regular graphs with connectivity~2 having $h_n^2(k)$ hamiltonian cycles much more efficiently than generating all $k$-regular graphs with $h_n^2(k)$ hamiltonian cycles and filtering them for connectivity~2 afterwards. For example: there are 8\,037\,418 connected 4-regular graphs on 16 vertices, but only 21\,247 of these graphs have connectivity~2.

We computed the values of $h_n(k)$ and $h_n^2(k)$ which are summarised in Table~\ref{correctedHaythorpeTable}. This computation took about 2~CPU years.

Haythorpe~\cite{Ha18} produced a table for $n\leq18$; our values $h_{19}^2(4)$ and $h_{19}^2(6)$ are new in comparison with~\cite{Ha18}.
However, it turns out that several values of $h_{n}^2(k)$ do not match with the numbers reported in~\cite{Ha18}. These values are underlined in Table~\ref{correctedHaythorpeTable}.

\begin{table}[h!] \centering
	\begin{threeparttable}
		\begin{tabular}{cccccccccccccccccccccc} \\
			\hline
			\noalign{\smallskip}
			 & $k=4$ & $k=5$ & $k=6$ & $k=7$\\
			\noalign{\smallskip}
			\hline
			\noalign{\smallskip}
			\multicolumn{1}{c}{$n=5$} & 12 (1) & - & - & -\\
			\multicolumn{1}{c}{$n=6$} & 16  (1) & 60 (1) & - & -\\
			\multicolumn{1}{c}{$n=7$} & 23 (1) & - & 360 (1) & -\\
			\multicolumn{1}{c}{$n=8$} & 29 (1) & 177 (1) & 744 (1) & 2520 (1)\\
			\multicolumn{1}{c}{$n=9$} & 36 (1) & - & 1553 (1) & -\\
			\multicolumn{1}{c}{$n=10$} & 36 (1) & 480 (1) & 3214 (1) & 14963 (1)\\
			\multicolumn{1}{c}{$n=11$} & 48 (2) & - & 6564 (1) & -\\
			\multicolumn{1}{c}{$n=12$} & 60 (2) & 576 (1) & 12000 (1) & 87808 (1)\\
			\multicolumn{1}{c}{$n=13$} & 72 (3) & - & 22680 (1) & -\\
			\multicolumn{1}{c}{$n=14$} & 72 (1) & 1296 (1) & 14400 (1) & 430920 (1)\\
			\multicolumn{1}{c}{$n=15$} & 72 (2) & - & $\mathbf{\underline{28800}}$ (1) & -\\
			\multicolumn{1}{c}{$n=16$} & 72 (1) & $\mathbf{\underline{3168}} (1) $ & $\mathbf{57600}$ (3) & $\mathbf{518400}$ (1)\\
			\multicolumn{1}{c}{$n=17$} & $\mathbf{96}$ (2) & - & $\mathbf{\underline{115200}}$ (2) & -\\
			\multicolumn{1}{c}{$n=18$} & $\mathbf{108}$ (1) & $\mathbf{3456}$ (1) & $\mathbf{\underline{230400}}$ (3) & $\mathbf{\underline{2663424}}$ (1)\\
			\multicolumn{1}{c}{$n=19$} & $\textbf{\textit{144}}$ (21) & - & $\textbf{\textit{443520}}$ (1) & -\\
			\noalign{\smallskip}
			\hline
		\end{tabular}
	\end{threeparttable}
	\caption{Counts of $h_n(k)$ and $h_n^2(k)$ for various values of $n$ and $k$. The non-bold values show $h_n(k)$, the bold values indicate $h_n^2(k)$ (i.e.~upper bounds for $h_n(k)$), and the underlined values represent all values that did not match with the values reported by Haythorpe in~\cite{Ha18}. The values which are new in comparison with~\cite{Ha18} are in italic. A dash signifies an impossible combination of $k$ and $n$. The number of graphs that have $h_n(k)$ or $h_n^2(k)$ hamiltonian cycles are indicated between parentheses; for the case of $h_n^2(k)$, only graphs with connectivity 2 are considered.}	
		\label{correctedHaythorpeTable}
\end{table}

Since several values do not match with the numbers previously reported in~\cite{Ha18}, it was important to take extra measures to ensure the correctness of the obtained results. For this reason, we implemented various ``sanity checks'' to check properties that should hold either by construction or to be in agreement with previously reported research. We verified for all graphs with $h_n(k)$ or $h_n^2(k)$ hamiltonian cycles that (i) the number of hamiltonian cycles is indeed equal to $h_n(k)$ or $h_n^2(k)$ according to the two independent programs; (ii) the reported values match the values reported in~\cite{Ha18} and~\cite{GMZ20} for the case of $h_n(k)$; (iii) $h_n(4) \leq h_n^2(4)$ (for all $n$ such that $10 \leq n \leq 16$), $h_n(5) \leq h_n^2(5)$ (for all $n$ such that $12 \leq n \leq 14$), and $h_{14}(6) \leq h_{14}^2(6)$; (iv) the graphs indeed have connectivity~2; and (v) the graphs are indeed $k$-regular graphs on $n$ vertices. Furthermore, we independently verified the values $h_n(k)$ by using GENREG followed by the two algorithms for counting the number of hamiltonian cycles and the values $h_{17}^2(4)$, $h_{16}^2(5)$ and $h_{15}^2(6)$ using GENREG and a program that filters out the graphs with connectivity~2 (and counting the number of hamiltonian cycles with the two algorithms).

Finally, we verified whether the number of hamiltonian cycles (or paths) was identical for both programs for all simple connected graphs up to order 11. As desired, all sanity checks produced the expected results. 

The complete lists of all $k$-regular graphs on $n$ vertices containing $h_n(k)$ or $h_n^2(k)$ hamiltonian cycles---whose counts are mentioned between parentheses in Table~\ref{correctedHaythorpeTable}---can be obtained from the database of interesting graphs from the \textit{House of Graphs}~\cite{HoG} by searching for the keywords ``regular * minimum number of hamiltonian cycles'' to allow other researchers to inspect or independently verify these results.

\section{Thomassen's independent dominating set method}
\label{sect:rigd}

It was shown by Thomassen~\cite{Th97} that any graph $G$ with a hamiltonian cycle ${\frak h}$ has a second hamiltonian cycle if $G$ has a vertex set $S$ which is an independent set of $(V(G), E({\frak h}))$ and is a dominating set of $(V(G), E(G) \setminus E({\frak h}))$; such a set $S$ will be called an \emph{${\frak h}$-independent dominating set} of $G$.  In fact, a slightly stronger statement is proven, which we state here:

\begin{theorem}[\cite{Th97}]\label{redgreen}
 Let $G$ be a hamiltonian graph, and let ${\frak h}$ be a hamiltonian cycle of $G$.  Colour all edges of ${\frak h}$ red, and all edges of $G-E({\frak h})$ green.  If $S$ is an ${\frak h}$-independent dominating set of $G$, then $G$ has a hamiltonian cycle distinct from ${\frak h}$.  Furthermore, we can choose a second hamiltonian cycle ${\frak h}'$ such that ${\frak h}'-S = {\frak h}-S$ and there is a vertex in $S$ such that exactly one red edge incident to it lies in ${\frak h}'$.
\end{theorem}

The proof of Theorem \ref{redgreen} relies on the following lemma, which can be proven by an extension of Thomason's lollipop method (see \cite{Th78}):

\begin{lemma}[\cite{Th97}]\label{partition}
 Let $H$ be a graph with a hamiltonian cycle $\h$.  Suppose there is a set of vertices $S$ such
that $H - S$ has $|S|$ components, each of which is a path with end vertices of odd degree in $H$. Then
\begin{itemize}
\item[\emph{(1)}] $\h' - S = \h - S$ for each hamiltonian cycle $\h'$ distinct from $\h$, and
\item[\emph{(2)}] each edge of $\h$ incident to a vertex in $S$ is in an even number of hamiltonian cycles.
\end{itemize}
\end{lemma}

Note that a graph $H$ containing an $\h$-independent dominating set $S$ satisfies the conditions of Lemma \ref{partition}.

In a subsequent article, Thomassen~\cite{Th98} ingeniously combines Theorem \ref{redgreen} with the Lov\'asz Local Lemma to prove that every $k$-regular hamiltonian graph with a hamiltonian cycle ${\frak h}$ always has an ${\frak h}$-independent dominating set if $k \ge 300$.
Thomassen's result was strengthened to $k \ge 24$ by Haxell, Seamone, and Verstraete~\cite{HSV07} by the same approach with an improved random process for selecting the vertices for a potential ${\frak h}$-independent dominating set.

In the language of our work, one can in fact show that regular graphs with sufficiently large degree not only contain two hamiltonian cycles, but in fact that $h(k) \geq 4$ if $k \ge 25$.  It is not hard to see that $h(k) \geq 3$ by an application of Lemma \ref{partition}.  Suppose now that there is some $k$-regular graph $H$ containing precisely three hamiltonian cycles, $\h_1, \h_2, \h_3$.  Note that $\h_1 \cup \h_2 \cup \h_3$ form a cubic spanning subgraph of $H$, that $\h_1 \cap \h_2$ form a perfect matching of $H$, and thus $H  - E(\h_1 \cap \h_2)$ is a hamiltonian $(k-1)$-regular graph; applying the aforementioned result of Haxell, Seamone, and Verstraete gives a second hamiltonian cycle distinct from $\h_1, \h_2, \h_3$, a contradiction.

In the remainder of this section, we first consider an extension of Thomassen's $\h$-independent dominating set approach by showing that not only do hamiltonian $k$-regular graphs contain multiple hamiltonian cycles if $k$ is sufficiently large, but also that one can find a second hamiltonian cycle in such graphs even if a linear number of edges of one hamiltonian cycle (in terms of $k$) are prescribed to be contained in the other. Furthermore, we explore the limits of the $\h$-independent dominating set method by showing that when $k$ is small (at most $6$), then one cannot guarantee that a hamiltonian graph contains an $\h$-independent dominating set; this answers questions of Thomassen~\cite{Th98} and Haxell et al.~\cite{HSV07} in the negative.  We also provide computational results related to the $k=7,8$ cases.

\subsection{An extension of Thomassen's ${\frak h}$-independent dominating set theorem}

The section is devoted to Theorem \ref{fixed edges}, a strengthening of Thomassen's Theorem~\ref{redgreen}. Our main tool is the Lov\'asz Local Lemma. 

\begin{theorem}[Lov\'asz Local Lemma~\cite{EL75}]\label{LLL}
	Let $\{A_i : i \in I\}$ be a finite family of events in a probability space, and for each $A_i$ let $J_i \subset I$ be the set of values of $j$ such that $A_j$ depends on $A_i$.  Suppose there exist real numbers $0 < x_i < 1$ for each $i \in I$ such that $$\mathbb{P}(A_i) < x_i \prod_{j \in J_i}(1-x_j).$$  Then $$\mathbb{P}\left(\bigcap_{i \in I} \overline{A_i}\right) \geq \prod_{i \in I}(1-x_i) > 0.$$
\end{theorem}

\begin{theorem}\label{fixed edges}
For any $\varepsilon > 0$, there exists a constant value $d_0(\varepsilon)$ such that the following holds: Let $H$ be a $d$-regular graph, $d \geq d_0(\varepsilon)$, containing a hamiltonian cycle ${\frak h}$, and let $E_f$ be a set of edges of ${\frak h}$ which cover at most $(1 - \varepsilon)d$ vertices of $H$.  Then $H$ contains a second hamiltonian cycle containing $E_f$.
\end{theorem}

\begin{proof}
We let $d_0 = d_0(\varepsilon)$ satisfy
\begin{eqnarray}
\frac{d_0}{4\sqrt{d_0}\log(8d_0^2)+1} > \frac{1}{\varepsilon}. \label{eqn:d}
\end{eqnarray}
Fix $\varepsilon>0$ and let $H$ be a $d$-regular graph, $d \geq d_0(\varepsilon)$, containing a hamiltonian cycle $\h$.  Colour every edge of $\h$ red and all other edges of $H$ green.  We prove the statement by applying the Local Lemma to a spanning subgraph $G$ of $H$ which contains $\h$ to show that $G$ (and hence $H$) contains an $\h$-independent dominating set.  The result then follows from Theorem \ref{redgreen}.

Fix a set of edges $E_f$ in $\h$, using at most $(1 - \varepsilon)d$ vertices of $H$.  Let $V_f$ be the set of ends of $E_f$ together with the vertices joined to an edge of $E_f$ by a red edge.  We seek an $\h$-independent dominating set $S \subseteq V(H) \setminus V_f$ which satisfies the conditions of Lemma~1; Theorem~4 then guarantees the existence of the desired cycle containing all edges of $E_f$.  Let $G$ be the subgraph of $H$ obtained by deleting all green edges incident to an edge in $E_f$.
Note that every vertex that is not an end of an edge in $E_f$ has degree at least $d-(1 - \varepsilon)d=\varepsilon d$.   By our choice of $d_0$, we have that each such vertex has degree strictly greater than $2$, since $$(1-\varepsilon)d < \left(1 - \frac{4\sqrt{d_0}\log(8d_0^2)+1}{d_0}\right)d < \left(1 - \frac{2}{d_0}\right)d = d-\frac{2d}{d_0} \leq d-2.$$
Let $p$ be some fixed real number such that $0 < p < 1$ and let $S \subseteq V(G) \setminus V_f$ be a set of vertices chosen independently at random with each vertex in $V(G) \setminus V_f$ being selected with probability $p$.
For each red edge of $G$ not incident to an edge of $E_f$, let $A_e$ be the event that each end of $e$ is in $S$.  For each vertex $v \in V(G)-V_f$, let $A_v$ be the event that $v$ and the vertices of $G$ joined to $v$ by green edges are all in $V(H)\setminus S$.  We then have the following probabilities: $$\mathbb{P}(A_e) = p^2,$$ $$(1-p)^{(d-2)+1} = (1-p)^{d-1} \leq \mathbb{P}(A_v) \leq (1-p)^{(d\varepsilon - 2)+1} = (1-p)^{d\varepsilon - 1}.$$
Let $I$ be the set of all vertices not in $V_f$ and all red edges other than those which have both ends in $V_f$.  Then $J_v = \{w \in I \cap V(G) : A_w \sim A_v\} \cup \{f \in I \cap E(G) : A_f \sim A_v\}$ and $J_e = \{w \in I \cap V(G) : A_w \sim A_e\} \cup \{f \in I \cap E(G) : A_f \sim A_e\}$, where we denote the dependency of events $A_i$ and $A_j$ by $A_i \sim A_j$.

Associate to every event of type $A_e$ the same variable $x \in (0,1)$ and to every event of type $A_v$ a variable $y \in (0,1)$.  Consider the event $A_e$.  For each end of $e$, there is precisely one other red edge incident to that end in $H$. Thus, there are at most two events of the form $A_{e_i}$ such that $A_{e_i} \sim A_e$.
For each end of $e$, there are at most $d-2$ vertices of $V(G) \setminus V_f$ joined to it by a green edge in $H$ and so there are at most $2(d-2)$ events $A_{v_i}$ such that $A_{v_i} \sim A_e$.  Also, if $e=uv$, then $A_u \sim A_e$ and $A_v \sim A_e$ and so there are at most $2d-2$ events $A_{v_i}$ such that $A_{v_i} \sim A_e$.  Thus, in order to apply the Local Lemma, the values of $x,y$ must satisfy
\begin{equation}\label{eqn:edge}
    \mathbb{P}(A_e) = p^2 < x(1-x)^2(1-y)^{2d-2}.
\end{equation}
Consider the event $A_v$.  There are exactly two red edges incident to each of the (at most) $d-2$ vertices joined to $v$ by green edges, as well as to $v$ itself, so there are at most $2(d-1)$ red edges incident to these $d-1$ vertices in total.  Thus, there are at most $2d-2$ events $A_{e_i}$ such that $A_{e_i} \sim A_v$.
Let $w$ be a vertex of $H$ joined to $v$ by a green edge.  Then $w$ is joined by a green edge to at most $d-3$ other vertices that lie in $V(G) \setminus V_f$.  These give at most $(d-3)+1$ (one for each of $w$ and its ``green neighbours'') events $A_{v_i}$ such that $A_{v_i} \sim A_v$.  Since this is true for each vertex joined to $v$ by a green edge, there are at most $(d-2)^2$ events $A_{v_i}$ such that $A_{v_i} \sim A_v$.  Thus, for each event $A_v$, we wish $x$ and $y$ to satisfy
\begin{equation}\label{eqn:vertex}
    \mathbb{P}(A_v) \leq (1-p)^{d\varepsilon - 1} < y(1-x)^{2d-2}(1-y)^{(d-2)^2}.
\end{equation}

Let $x = \frac{1}{4d}, y = \frac{1}{2d^2}, p = \frac{1}{4\sqrt{d}}$.  Then (\ref{eqn:edge}) holds because
\begin{eqnarray*}
p^2 &=& \frac{1}{16d}
< \left(\frac{1}{4d}\right)\left(1-\frac{1}{4d}\right)^2\left(\frac{1}{e}\right) \\
&<& \left(\frac{1}{4d}\right)\left(1-\frac{1}{4d}\right)^2\left(1-\frac{1}{2d^2}\right)^{2d^2-1}  \\
&<& \left(\frac{1}{4d}\right)\left(1-\frac{1}{4d}\right)^2\left(1-\frac{1}{2d^2}\right)^{2d-2} \\
&=& x(1-x)^2(1-y)^{2d-2}.
\end{eqnarray*}
Before moving on to showing the second inequality holds, notice that from (\ref{eqn:d}) we have

$$\frac{d_0}{4\sqrt{d_0}\log(8d_0^2)+1} > \frac{1}{\varepsilon},$$
whence
$$\frac{1-{\varepsilon}d}{4\sqrt{d}} < \log\left(\frac{1}{8d^2}\right)\hspace{-0.5mm}.$$
Thus (\ref{eqn:vertex}) holds since
\begin{eqnarray*}
(1-p)^{d\varepsilon - 1} &\leq& e^{-p(d\varepsilon - 1)} = e^{\frac{1-{\varepsilon}d}{4\sqrt{d}}} \\
&<& e^{\log(\frac{1}{8d^2})} = \frac{1}{8d^2} \\
&<& \left(\frac{1}{2d^2}\right)\left(1-\frac{2d-2}{4d}\right)\left(1-\frac{(d-2)^2}{2d^2}\right) \\
&<& \left(\frac{1}{2d^2}\right)\left(1-\frac{1}{4d}\right)^{2d-2}\left(1-\frac{1}{2d^2}\right)^{(d-2)^2} \\
&=& y(1-x)^{2d-2}(1-y)^{(d-2)^2}.
\end{eqnarray*}
So, $G$ contains an $\h$-independent dominating set by the Local Lemma which contains all edges of $E_f$, and thus so does $H$.
\end{proof}

\subsection{Limitations of the ${\frak h}$-independent dominating set method}

Thomassen not only proves Theorem \ref{redgreen} in \cite{Th98}, but also the following much more wide ranging statement:

\begin{theorem}[\cite{Th98}]\label{theorem:rigd}
Let $H$ be a graph whose edges are partitioned into red and green subgraphs $R$ and $K$, respectively.  If $R$ is $r$-regular and $K$ is $k$-regular, with $r \geq 3$ and $k > 200\hspace{0.6mm}r\ln{r}$, then $H$ contains a red-independent green-dominating set.
\end{theorem}

Ghandehari and Hatami state in a technical report that the conditions in Theorem \ref{theorem:rigd} are near-optimal:

\begin{theorem}[\cite{GH}]\label{theorem:optimal}
For every $r, k > 0$ with $k \leq r \ln r$ there exists a graph $H$ whose edges are coloured red or green, such that the spanning red subgraph is $r$-regular, and the spanning green subgraph is $k$-regular, and $H$ does not have a red-independent green-dominating set.
\end{theorem}

However, as this result is meaningful only for $r \geq 3$ (and $k$ sufficiently large), the limits of the $\h$-independent dominating set approach for guaranteeing a second hamiltonian cycle are worth deeper exploration.

Thomassen~\cite{Th98} first observed that it is impossible to prove Sheehan's conjecture by the $\h$-independent dominating set approach, as there are infinitely many 4-regular hamiltonian graphs having no independent dominating set with respect to some prescribed hamiltonian cycle. In fact, it is straightforward to see that the complete graph on five vertices is the smallest such example. As it is known that every hamiltonian $k$-regular graph, where $k \ge 3$ is an odd integer, contains at least three hamiltonian cycles~\cite{Th78}, Thomassen poses the following problem:

\begin{ques}[\cite{Th98}]\label{q:6-reg}
Does every $6$-regular graph containing a hamiltonian cycle $\h$ contain a $\h$-independent dominating set?
\end{ques}

In \cite{HSV07}, the existence of further infinite families of hamiltonian $4$-regular graphs without $\h$-independent dominating sets is shown, though all have small girth.  As such, the authors suggest the following questions:

\begin{ques}[\cite{HSV07}]\label{q:5-reg}
Does there exist a $5$-regular graph with a hamiltonian cycle $\h$ containing no $\h$-independent dominating set?
\end{ques}

\begin{ques}[\cite{HSV07}]\label{q:4-reg+girth}
Does every $4$-regular graph of sufficiently large girth with a hamiltonian cycle $\h$ have a $\h$-independent dominating set?
\end{ques}

Motivated by these questions, we prove in this section the existence of 5- and 6-regular graphs that have a hamiltonian cycle ${\frak h}$ but no ${\frak h}$-independent dominating set, thus answering Questions \ref{q:6-reg} and \ref{q:5-reg}.
Indeed, to obtain an infinite family of desired graphs, it suffices to construct one such graph. 
\begin{lemma}
If there exists a $k$-regular graph $H$ containing a hamiltonian cycle $\mathfrak{h}_H$ without $\mathfrak{h}_H$-independent dominating sets, then there exist infinitely many $k$-regular graphs containing a hamiltonian cycle $\mathfrak{h}$ without $\mathfrak{h}$-independent dominating sets.
\end{lemma}
\begin{proof}
	Let $e$ be an edge in $\mathfrak{h}_H$, and let $H'$ and $\mathfrak{h}_H'$ be obtained from $H$ and $\mathfrak{h}_H$ by deleting $e$, respectively. Then every set that is dominating in $(V(H'), E(H') \setminus E(\mathfrak{h}_H')) = (V(H), E(H) \setminus E(\mathfrak{h}_H))$ and independent in $(V(H'), E(\mathfrak{h}_H'))$ must contain both end-vertices of $\mathfrak{h}_H'$. It is obvious that we can take arbitrarily many copies of $H'$ and join them by independent edges to obtain a $k$-regular graph $G$ with a hamiltonian cycle $\mathfrak{h}$ such that $\mathfrak{h}$ contains all copies of $\mathfrak{h}_H'$. Any dominating set of $(V(G), E(G) \setminus E({\frak h}))$ contains all end-vertices of the new edges used to construct $G$, and hence is not an independent set of $(V(G), E({\frak h}))$.
\end{proof}

In the next subsections we present our results on $k$-regular graphs with a hamiltonian cycle ${\frak h}$ containing no ${\frak h}$-independent dominating set for $5 \leq k \leq 8$. These results were obtained by a computer search where we first use the generator GENREG~\cite{Me99} to construct all $(k-2)$-regular graphs on $n$ vertices. On these generated graphs, we apply a program which computes all minimal dominating sets of these graphs, enumerates all hamiltonian cycles in the complement graphs and checks whether any hamiltonian cycle ${\frak h}$ can be added to the original graph such that no minimal dominating set is ${\frak h}$-independent.

\subsubsection{The 5-regular case: Solving Question~\ref{q:5-reg} of Haxell et al.\ }

By a computer search using the approach described above, we determined the following:

\begin{proposition}
There exist infinitely many $5$-regular graphs with a hamiltonian cycle ${\frak h}$ containing no ${\frak h}$-independent dominating set.
\end{proposition}

The smallest such graphs $G$ have order 10, one of which is given in Fig.~\ref{fig:5reg10}, where the dashed edges denote the hamiltonian cycle ${\frak h}$. It is straightforward to show that the set of minimal dominating sets of $(V(G), E(G) \setminus E(\frak{h}))$ is $\mathcal{D}_3 \cup \mathcal{D}_4 \cup \mathcal{D}_5$, where $\mathcal{D}_3 := \{\{a_i, b_i, b_{i - 1}\}, \{a_i, a_{i + 1}, b_i\} : i = 1, \dots, 5\}\}$, $\mathcal{D}_4 := \{\{a_i, a_{i + 2}, b_j, b_{j + 2}\} : i, j = 1, \dots, 5\}$, and $\mathcal{D}_5 := \{\{a_1, a_2, a_3, a_4, a_5\}, \{b_1, b_2, b_3, b_4, b_5\}\}$ (all indices are taken modulo 5). Then it is easy to verify that every dominating set of $(V(G), E(G) \setminus E(\frak{h}))$ must induce some edge in $(V(G), E(\frak{h}))$ and $G$ has no ${\frak h}$-independent dominating set.

\begin{figure}[h!]
	\begin{center}
		\includegraphics[scale=1.2]{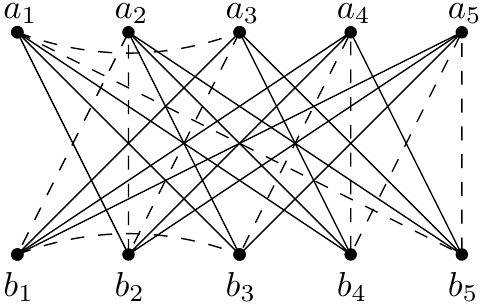}
		\caption{A 5-regular graph on 10 vertices with a hamiltonian cycle $\frak{h}$ (dashed edges) having no $\frak{h}$-independent  dominating set.}\label{fig:5reg10}
	\end{center}
\end{figure}

Table~\ref{table:counts_noRIGDSet_5reg} in the Supplementary material lists the number of pairs $(G',\frak h)$ where $G'$ is a 3-regular graph and ${\frak h}$ is a hamiltonian cycle in its complement such that no minimal dominating set in $G'$ is ${\frak h}$-independent. 

We now present a graph $G$ that has a hamiltonian path $\frak{p}$ such that all dominating sets of $(V(G), E(G) \setminus E(\mathfrak{p}))$ are not independent in $(V(G), E(\frak{p}))$, the end-vertices of $\frak{p}$ are not adjacent, and $G$ becomes 5-regular if they are joined by an edge. Such a graph is given in Fig.~\ref{fig:5reg12inf}, where the dashed edges denote a hamiltonian path $\frak{p}$. It is straightforward to show that the set of minimal dominating sets of $(V(G), E(G) \setminus E(\frak{p}))$ is $\mathcal{D}_4 \cup \mathcal{D}_5 \cup \mathcal{D}_6$, where
$\mathcal{D}_4 := \{\{b_i, c_{i + 1}, c_{i + 2}, d_j\}, \{b_i, c_{i + 1}, c_{i + 2}, e\}, \{b_i, b_{i + 1}, c_{i + 2}, d_{i + 2}\}, \{b_i, b_{i + 1}, c_{i + 2}, e\}, \{a_1, a_2, c_i, d_i\} : i, j = 1, 2, 3\} \cup \{\{b_1, b_2, b_3, e\}\} \cup \{\{a_i, b_j, d_j, e\}, \{a_i, b_j, c_k, d_k\} : i = 1, 2 \textrm{ and } j, k = 1, 2, 3\}$,
$\mathcal{D}_5 := \{\{a_i, b_j, d_{j + 1}, d_{j + 2}, e\}, \{a_i, b_j, d_1, d_2, d_3\}, \{a_1, a_2, d_j, d_{j + 1}, e\} : i = 1, 2 \textrm{ and } j = 1, 2, 3\} \cup \{\{a_1, a_2, $ $d_1, d_2, d_3\}\}$, and
$\mathcal{D}_6 := \{\{a_1, a_2, c_1, c_2, c_3, e\}, \{b_1, b_2, b_3, d_1, d_2, d_3\}\}$ (all indices are taken modulo 3). It is not hard to verify that every set from $\mathcal{D}_4 \cup \mathcal{D}_5 \cup \mathcal{D}_6$ is not independent in $(V(G), E(\frak{p}))$.

\begin{figure}[h!]
	\begin{center}
		\includegraphics[scale=1.2]{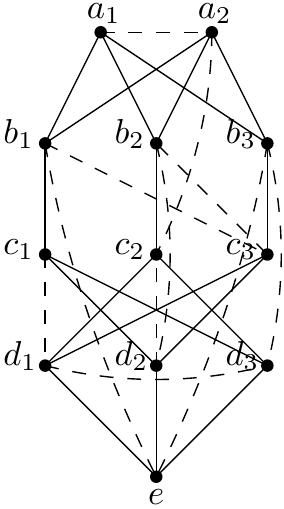}
		\caption{A graph $G$ on 12 vertices with a hamiltonian path $\frak{p}$ (dashed edges) that becomes 5-regular if the edge $a_1 c_1$ is added. Every vertex set of $G$ is not dominating in $(V(G), E(G) \setminus E(\frak{p}))$ or not independent in $(V(G), E(\frak{p}))$.}\label{fig:5reg12inf}
	\end{center}
\end{figure}

\subsubsection{The 6-regular case: Solving the first open case of Thomassen's Question~\ref{q:6-reg}}

Combining theoretical arguments and an exhaustive computer search, we established the following:

\begin{proposition}
There exist infinitely many $6$-regular graphs with a hamiltonian cycle ${\frak h}$ containing no ${\frak h}$-independent dominating set.
\end{proposition}

Along the way, we determined all such graphs up to 15 vertices. The counts are summarised in Table~\ref{table:counts_noRIGDSet_6reg} in the Supplementary material. The last column of this table lists the number of pairs $(G',\frak h)$ where $G'$ is a 4-regular \textit{bipartite} graph and ${\frak h}$ is a hamiltonian cycle in its complement such that no minimal dominating set in $G'$ is ${\frak h}$-independent. The reason why we also specifically looked into bipartite graphs $G'$ is that it follows from Table~\ref{table:counts_noRIGDSet_5reg} that in the 5-regular case there are examples of minimal order where $G'$ is bipartite and restricting the search to bipartite graphs allowed us to go several orders further. These computations required about 10 CPU years and the pairs from Table~\ref{table:counts_noRIGDSet_6reg} can be downloaded from~\cite{code}. 

One of the 6-regular graphs $G$ with a hamiltonian cycle $\frak{h}$ having no $\frak{h}$-independent dominating set is given in Fig.~\ref{fig:6reg15}, where the dashed edges denote a hamiltonian cycle $\frak{h}$ of $G$. 

Denote by $S_3$ the permutation group on $\{1, 2, 3\}$. It is straightforward to show that the set of minimal dominating sets of $(V(G), E(G) \setminus E(\frak{h}))$ is $\mathcal{D}_3 \cup \mathcal{D}_4 \cup \mathcal{D}_5 \cup \mathcal{D}_6$, where
$\mathcal{D}_3 := \{\{a_{1i}, a_{2i}, a_{3i}\} : i = 1, 2\} \cup \{\{b_{1\sigma(1)}, b_{2\sigma(2)}, b_{3\sigma(3)}\} : \sigma \in S_3\}$,
$\mathcal{D}_4 := \{\{a_{i1}, a_{i2}, a_{i+1, 2}, a_{i+2, 1}\} : i = 1, 2, 3\} \cup \{\{b_{ij}, b_{i, j+1}, b_{i+1, j+2}, b_{i+2, j+2}\} : i, j = 1, 2, 3\}$,
$\mathcal{D}_5 := \{\{a_{ij}, a_{i+1, k}, b_{i+1, \ell}, b_{i+2, \ell+1}, b_{i+2, \ell+2}\},$  $\{a_{ij}, a_{i+1, k}, b_{i+1, \ell}, b_{i+1, \ell+1}, b_{i+2, \ell+2}\} : i, \ell = 1, 2, 3 \textrm{ and } j, k = 1, 2\} \cup \{\{a_{i1}, a_{i2}, a_{i+1, 2}, b_{i+1, j}, b_{i+2, k}\},$ $\{a_{i1}, a_{i+1, 1}, a_{i+1, 2}, b_{i+1, j}, b_{i+2, k}\} : i, j, k = 1, 2, 3\}$,
and
$\mathcal{D}_6 := \{\{a_{i1}, a_{i+1, j}, a_{i+2, 2}, b_{k1}, b_{k2}, b_{k3}\} : i, k = 1, 2, 3 \textrm{ and } j = 1, 2\}$ (all indices are taken modulo 3). It is not hard to verify that every set from $\mathcal{D}_3 \cup \mathcal{D}_4 \cup \mathcal{D}_5 \cup \mathcal{D}_6$ is not independent in $(V(G), E(\frak{h}))$ and hence $G$ has no $\frak{h}$-independent dominating set.

\begin{figure}[h!]
	\begin{center}
		\includegraphics[scale=1.2]{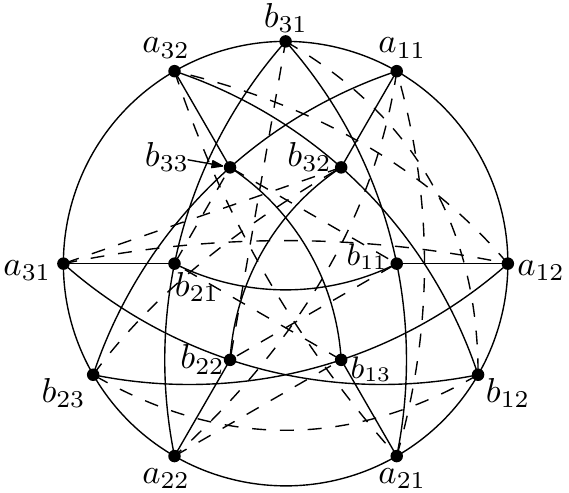}\quad\includegraphics[scale=1.2]{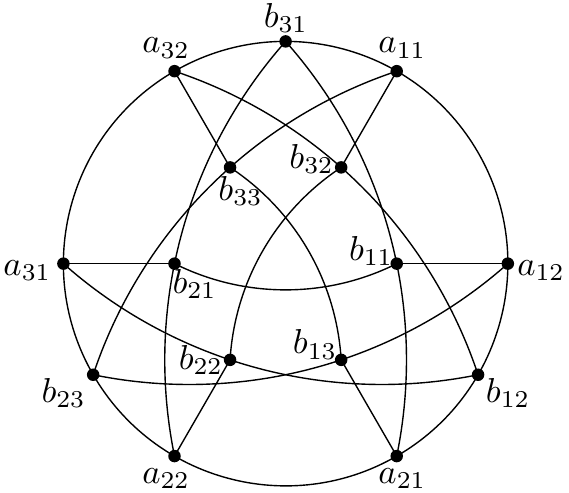}
		\caption{On the left-hand side, a 6-regular graph on 15 vertices with a hamiltonian cycle~$\frak{h}$ (dashed edges) having no $\frak{h}$-independent dominating set is shown. The right-hand side depicts the same graph without $E(\frak{h})$.}\label{fig:6reg15}
	\end{center}
\end{figure}

We append one more graph which is 6-regular, has a hamiltonian cycle, yet has no associated dominating independent sets. The graph is given in Fig.~\ref{fig:6reg16}, with the details omitted.

\begin{figure}[h!]
	\begin{center}
		\includegraphics[scale=1.2]{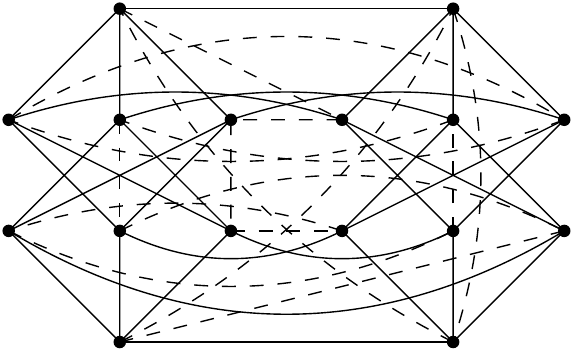}\quad\includegraphics[scale=1.2]{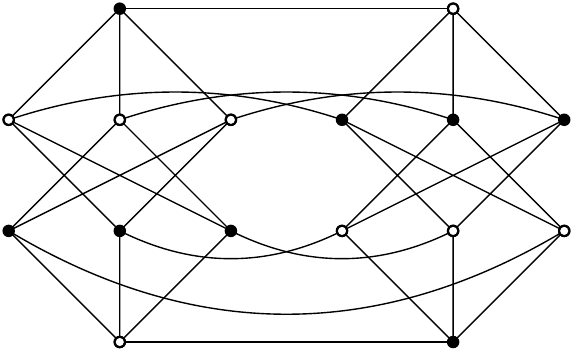}
		\caption{On the left-hand side, a 6-regular graph on 16 vertices with a hamiltonian cycle (dashed edges) having no associated independent dominating set is shown. The right-hand side depicts the same graph without the hamiltonian cycle; we note that the resulting graph is bipartite.}\label{fig:6reg16}
	\end{center}
\end{figure}

\subsubsection{The 7- and 8-regular case}

Concluding this section, we performed exhaustive computer searches for the 7- and 8-regular case but no such graphs containing a hamiltonian cycle $\frak{h}$ having no $\frak{h}$-independent dominating set were found. These computations are summarised in the following observation. In total this required approximately 15 CPU years.

\begin{observation}
There are neither $7$- nor $8$-regular graphs (bipartite graphs) containing a hamiltonian cycle $\frak{h}$ and having no $\frak{h}$-independent dominating set of order at most $14$ (at most $20$).
\end{observation}

\section{\boldmath$(k,\ell)$-regular graphs}
\label{sect:kl_regular}

The next observation follows directly from a result of Thomason~\cite{Th78}.

\begin{observation}
\label{oddDegreeObservation}
For every positive odd (not necessarily distinct) integers $k \ge 3$ and $\ell \ge 3$, we have $h(k,\ell) \ge 3$.
\end{observation}

If we allow 2-valent vertices, the situation is easy to describe. For an integer $k \ge 2$, consider a $k$-regular graph containing a hamiltonian cycle ${\frak h}$. Insert on all (or indeed all but one) edges of ${\frak h}$ a 2-valent vertex. The resulting graph contains only vertices of degree 2 or $k$, at least one of each, and exactly one hamiltonian cycle. Thus, we have:

\begin{observation}
\label{(2,q)-regularUHG}
For every integer $k \ge 2$ and every $n$ such that there exists a hamiltonian $k$-regular graph of order $n$, we have $h_{2n-1}(2,k)=h_{2n}(2,k)=1$.
\end{observation}

We now shift our attention to the semi-cubic case, i.e.\ the values of $h(3,k)$. The case $k = 2$ has been discussed above, and for $k = 3$ we have $h_n(3) = 3$ for all even $n \ge 4$: it follows from Observation~\ref{oddDegreeObservation} that $h_n(3) \ge 3$ and by considering $K_4$ and successively replacing vertices by triangles, one infers $h_n(3) \le 3$. Henceforth, a graph containing exactly one hamiltonian cycle shall be called \emph{uniquely hamiltonian}. For $k = 4$, Entringer and Swart provided a uniquely hamiltonian graph of connectivity~2, see~\cite{ES80}, and Aldred and Thomassen described a 3-connected such graph in~\cite{HA99}. We recall that $h(3,2k+1) \geq 3$ by Observation~\ref{oddDegreeObservation}. We now treat the $(3,2k)$-regular case; the proof of the next statement was described to us by Gunnar Brinkmann~\cite{Br22}.

\begin{theorem}\label{(3,2k)-regularUHG}
For every $\kappa \in \{ 2, 3 \}$ and any positive integer $k$, there are infinitely many $(3,2k)$-regular uniquely hamiltonian graphs of connectivity~$\kappa$.
\end{theorem}
\begin{proof}
We have dealt with the case $k = 1$ above, so assume henceforth $k > 1$. We prove the case $\kappa = 3$ since the case $\kappa = 2$ is a direct corollary. The following auxiliary result will be useful; its proof is straightforward and left to the reader.

\smallskip

\noindent \textsc{Claim} 1. \emph{Let $P$ be the Petersen graph and $x \in V(P)$ with neighbours $x_1, x_2, x_3$. Then $P - x$ admits no hamiltonian $x_ix_j$-path and $P - x - x_i$ has exactly two hamiltonian $x_jx_k$-paths, where $\{ i,j,k \} = \{ 1, 2, 3 \}$.}

\smallskip

For a positive integer $k$, let $G$ be a $(3,2k)$-regular multigraph with exactly two hamiltonian cycles and exactly one vertex $x$ of degree $2k$. Let $vw \in E(G)$ be such that exactly one of the hamiltonian cycles in $G$ traverses it, and $x$ is not incident to $vw$. Let $H$ be obtained from the Petersen graph by removing a vertex---we call its neighbours $v', w', x'$. Consider the graph obtained from the disjoint union of $G - vw$ and $H$ by joining $v$ to $v'$ and $w$ to $w'$, and identifying $x$ and $x'$; this operation will be called $(\dagger)$. Thus, by Claim~1, a $(3,2k+2)$-regular multigraph with exactly two hamiltonian cycles and exactly one vertex of degree $2k + 2$ is obtained.

\smallskip

\noindent \textsc{Claim} 2. \emph{There exists a $3$-connected $(3,4)$-regular graph with exactly two hamiltonian cycles and exactly one vertex of degree~$4$.}

\smallskip

\noindent \emph{Proof of Claim 2.} This follows by applying $(\dagger)$ to a triangle with one edge doubled. Alternatively, Aldred and Thomassen describe a (larger) such graph in~\cite{HA99}.

\smallskip

\noindent \textsc{Claim} 3. \emph{If there exists a $3$-connected $(3,2k)$-regular graph $G$ with exactly two hamiltonian cycles and exactly one non-cubic vertex, then there are infinitely many $3$-connected $(3,2k)$-regular uniquely hamiltonian graphs.}

\smallskip

\noindent \emph{Proof of Claim 3.} How to produce from $G$ a 3-connected $(3,2k)$-regular graph $H$ with exactly one hamiltonian cycle is not difficult and explained in~\cite{HA99}; we recall the argument for the reader's convenience. There must exist a cubic vertex in $G$ with incident edges $a,b,c$ such that $G$ has a unique hamiltonian cycle through the edges $a,c$ and a unique hamiltonian cycle through the edges $b,c$ but no hamiltonian cycle through $a,b$. We take two copies of $G$ and label the edges $a,b,c$ on one graph and label the corresponding edges $a',b',c'$, on the other graph. Let $v'$ be the vertex incident with $a',b',c'$. Remove $v$ and $v'$ and join the ``hanging'' edges as follows: $a$ to $c'$, $b$ to $a'$, and $c$ to $b'$. To visit all vertices in the first copy of $G$ we have to use the edges $a,c$ or $b,c$. Since the
latter edges are joined to $a',b'$ in the second copy of $G$, we cannot visit all vertices of the second copy. However, going through $a,c$ in the first copy of $G$ to $b',c'$ in the second copy provides the unique hamiltonian cycle. 
Iteratively replacing in $H$ cubic vertices by triangles yields the second statement. This completes the proof of the claim.

\smallskip

By Claims 2 and 3, the proof of the case $\kappa = 3$ is complete. For the connectivity 2 case, consider a 3-connected $(3,2k)$-regular graph $G$ containing exactly one hamiltonian cycle ${\frak h}$. Let $vw \in E({\frak h})$. Let $G'$ be a copy of $G$ with $v' \in V(G')$ corresponding to $v$ and $w' \in V(G')$ corresponding to $w$. Adding to the disjoint union of $G - vw$ and $G' - v'w'$ the edges $vv'$ and $ww'$ yields the statement. \end{proof}

Let $G$ be a cubic hamiltonian graph and ${\frak h}$ some hamiltonian cycle in $G$. The next argument uses the same basic idea as described in the proof of Claim~3. For every vertex $v$ of $G$, there are exactly two edges $a, b$ incident with $v$ and on ${\frak h}$. Let the third edge incident with $v$ be $c$. Consider a $(3,2k)$-regular uniquely hamiltonian graph $H$ and therein a vertex $w$. There are exactly two edges $a', b'$ incident with $w$ and on the hamiltonian cycle of $H$. Let the third edge incident with $w$ be $c'$. Consider $G - v$ and $H - w$, but leave the edges $a,b,c,a',b',c'$ as dangling half-edges. In the disjoint union of $G - v$ and $H - w$, connect $a$ with $a'$; $b$ with $b'$; and $c$ with $c'$. Performing this replacement for all (or indeed all but one) vertices of $G$, we obtain a $(3,2k)$-regular uniquely hamiltonian graph.

Furthermore, we remark that with similar techniques, it can be shown that for any $\kappa \in \{ 2, 3 \}$ and any set $M$ of numbers containing the number 3 and at least one positive even number, there exist infinitely many uniquely hamiltonian graphs of connectivity $\kappa$ for which $M$ is its set of vertex degrees.

A natural question is to determine the order of the smallest $(3,k)$-regular hamiltonian graphs with the fewest hamiltonian cycles for small values of $k$. The smallest open case is $k = 4$. It follows from Theorem~\ref{(3,2k)-regularUHG} that $h(3,2k) = 1$.

The generation algorithm for uniquely hamiltonian graphs from~\cite{GMZ20} can be efficiently restricted to generate uniquely hamiltonian graphs with a given maximum degree. By restricting the algorithm to only generate $(3,4)$-regular graphs, we determined that the smallest uniquely hamiltonian $(3,4)$-regular graphs have 18 vertices. The results are summarised in Observation~\ref{obs:UH_34}. This computation required approximately 30 CPU days and the results were independently verified up to 15 vertices by using GENG~\citep{MP14} to generate $(3,4)$-regular graphs and testing if the generated graphs are uniquely hamiltonian are not. Moreover, we also verified, using an independent program, that the five $(3,4)$-regular graphs on 18 vertices from Observation~\ref{obs:UH_34} indeed only have one hamiltonian cycle.

\begin{observation}\label{obs:UH_34}
The smallest uniquely hamiltonian $(3,4)$-regular have $18$ vertices and there are exactly five such graphs of that order.
\end{observation}

The uniquely hamiltonian $(3,4)$-regular graphs from Observation~\ref{obs:UH_34} are shown in Fig.~\ref{fig:UH_34} in the Supplementary material and can also be inspected in the database of interesting graphs from the \textit{House of Graphs}~\cite{HoG} by searching for the keywords ``uniquely hamiltonian (3,4)-graph''.

By using the same approach for searching for uniquely hamiltonian $(3,6)$-graphs we were not able to determine the smallest such graphs, but were able to give the following lower bound. This computation required approximately 8 CPU years. 

\begin{observation}\label{obs:UH_36}
The smallest uniquely hamiltonian $(3,6)$-regular graphs have at least $17$ vertices.
\end{observation}

We first focussed on the even case, for which we know that $h(3,2k) = 1$, and now move to the odd case for which we know that $3 \leq h(3,2k+1)$ and will show that $h(3,2k+1) \leq 4$.

Let $G$ be a $(3,2k+1)$-regular graph $G$ with exactly three hamiltonian cycles. Thomason~\cite{Th78} showed that in a graph in which all vertices have odd degree, every edge is traversed by an even number of hamiltonian cycles. Thus, around every vertex in $G$ exactly three of its incident edges are each traversed by exactly two distinct hamiltonian cycles---no other distribution is possible. Hence, a $(3,2k+1)$-regular graph with exactly three hamiltonian cycles must contain a cubic spanning subgraph with three hamiltonian cycles. All cubic graphs up to 32 vertices with exactly three hamiltonian cycles have been described, see Table~7 of~\cite{GMZ20}. By starting from these graphs, we determined that for each of these graphs, every pair of non-adjacent vertices has a hamiltonian path between them. Therefore, we have:

\begin{observation}
All $(3,2k+1)$-regular graphs with exactly three hamiltonian cycles and at most~$32$ vertices are cubic.
\end{observation}

By Observation~\ref{oddDegreeObservation}, for every positive integer $k$, every hamiltonian $(3,2k+1)$-regular graph contains at least three hamiltonian cycles. We now prove that the minimum number of hamiltonian cycles in such graphs is at most four and we will show that the graphs described in our proof are the smallest $(3,2k+1)$-regular graphs with exactly four hamiltonian cycles. It would be interesting to determine, for $k > 1$, whether $h(3,2k+1)$ is 3 or 4.

\begin{proposition}\label{(3,2k+1)-regular}
For any $\kappa \in \{ 2, 3 \}$ and every integer $k \ge 2$ there exists a hamiltonian $(3,2k+1)$-regular graph of connectivity~$\kappa$ containing at most four hamiltonian cycles. For $k=2$, the smallest $(3,5)$-regular graph with four hamiltonian cycles has order $6$ and for every integer $k \ge 3$, the smallest $(3,2k+1)$-regular graph with four hamiltonian cycles has order $2k+4$. For every integer $k \geq 2$ there exists a graph of connectivity~$2$ such that this graph is a minimal order $(3,2k+1)$-regular graph with four hamiltonian cycles.
\end{proposition}

\begin{proof}
The lower bound is given by Observation~\ref{oddDegreeObservation}.

For the connectivity~2 case and $k=2$, Fig.~\ref{h(3,5)} shows a $(3,5)$-regular graph on 6 vertices of connectivity~2 containing precisely four hamiltonian cycles (this order is the smallest possible order, because the smallest $(3,5)$-regular graph has order 6). For the connectivity~2 case and $k \geq 3$, let $H$ be the $(3,2k+1)$-regular graph on $2k+4$ vertices shown in Fig.~\ref{h(3,2k+1)}. Note that every hamiltonian cycle of $H$ contains the edges $a_0c$ and $b_0c$, because $H-c-d$ is disconnected. The graph $H-c$ has precisely four hamiltonian paths between $a_0$ and $b_0$ such that $H$ has precisely four hamiltonian cycles. Since $H$ is hamiltonian and $H-c-d$ is disconnected, $H$ indeed has connectivity~2. To show the minimality of $H$, let $G$ be a $(3,2k+1)$-regular graph on $2k+2$ vertices. We will now show that $G$ has either zero or strictly more than four hamiltonian cycles based on the following case distinction.

If $G$ has precisely one vertex $v$ of degree $2k+1$, then $G-v$ is a 2-regular graph on $2k+1 \geq 7$ vertices. If $G-v$ is disconnected, $G$ does not have any hamiltonian cycles and otherwise $G$ has precisely $2k+1$ hamiltonian cycles.

If $G$ has precisely two (three) vertices $v$, $w$ (and $x$), then $G-v-w$ ($G-v-w-x$) has $2k$ ($2k-1$) vertices of degree~1 (degree~0). For $k=3$, we verified by computer that $h_{8}(3,7) = 7$. For $k \geq 4$, note that any hamiltonian cycle in $G$ contains at most two consecutive vertices in $G-v-w$ ($G-v-w-x$). However, there are $2k \geq 8$ ($2k-1 \geq 7$) vertices in $G-v-w$ ($G-v-w-x$) such that $G$ cannot have any hamiltonian cycles.

Finally, $G$ cannot have more than three vertices of degree $2k+1$, because $G$ is a $(3,2k+1)$-regular graph on $2k+2$ vertices.

There are no $(3,2k+1)$-regular graphs on $2k+3$ vertices, because the sum of the degrees has to be even. This proves the minimality of $H$.

For the connectivity~3 case, consider a cubic graph $G$ with exactly three hamiltonian cycles---in such a graph every edge is traversed by exactly two hamiltonian cycles~\cite{Th78}. We apply $(\dagger)$ (as defined in the proof of Theorem~\ref{(3,2k)-regularUHG}) to any edge of $G$ and any vertex, and obtain a 3-connected $(3,5)$-regular graph containing exactly four hamiltonian cycles. In such a graph, every edge is traversed by an even number of hamiltonian cycles~\cite{Th78}, so there must exist an edge traversed by exactly two hamiltonian cycles. Therefore, we can iterate this procedure until a 3-connected $(3,2k+1)$-regular graph containing exactly four hamiltonian cycles is obtained.
\end{proof}

\begin{figure}[h!]
\centering
\begin{subfigure}{0.4\linewidth}
\includegraphics[height=40mm]{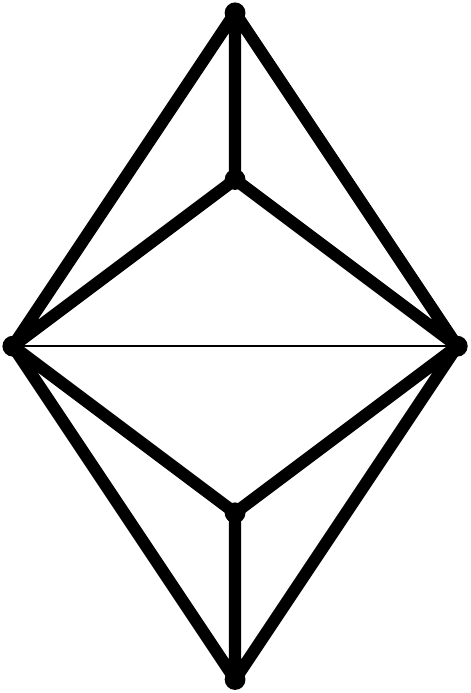}
\caption{}\label{h(3,5)}
\label{t1Subfigure}
\end{subfigure}%
\begin{subfigure}{0.4\linewidth}
  \includegraphics[height=80mm]{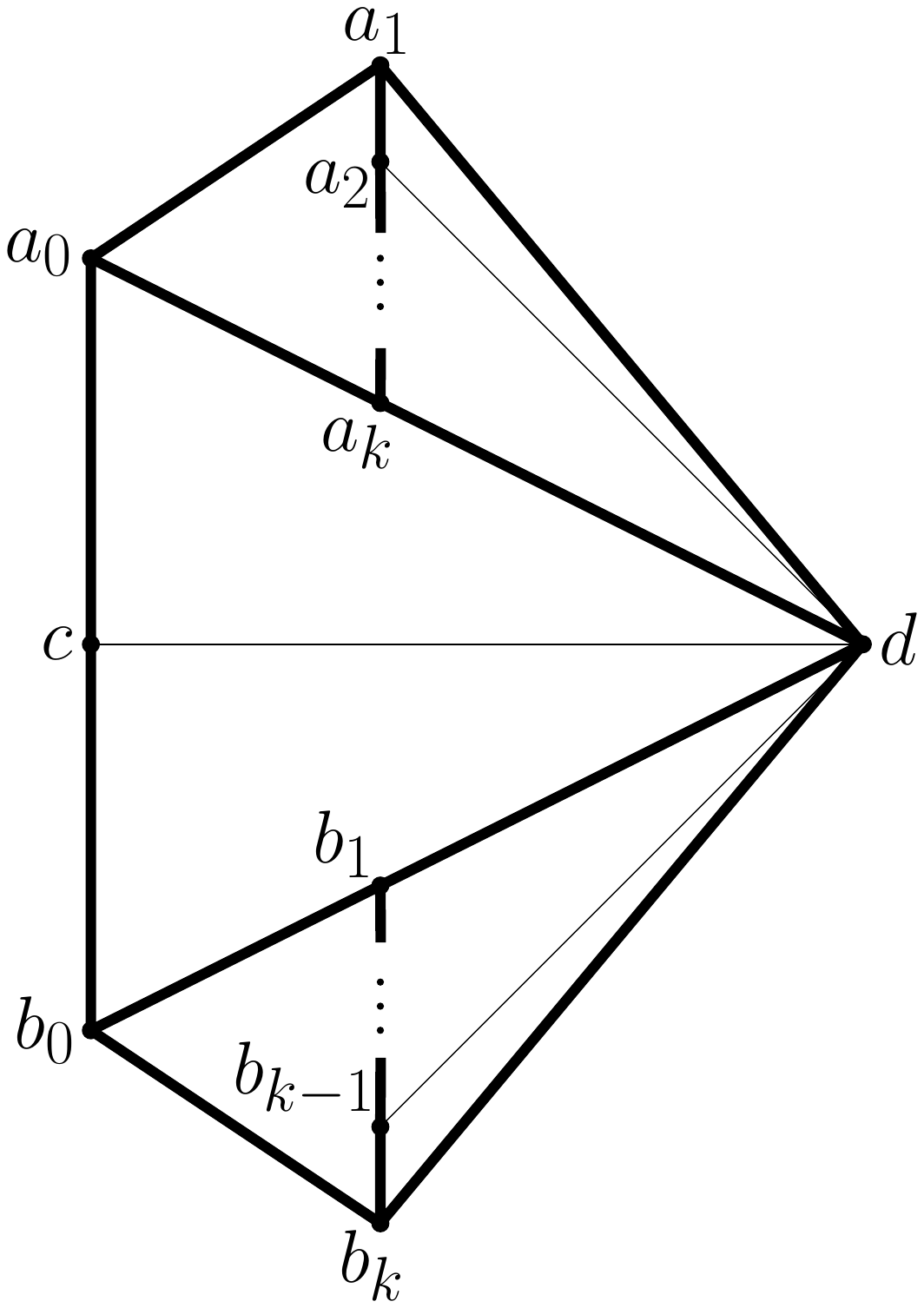}
  \caption{}\label{h(3,2k+1)}
\end{subfigure}
\caption{(a) A $(3,5)$-regular graph of connectivity~2 containing precisely four hamiltonian cycles. (b) A $(3,2k+1)$-regular graph of connectivity~2 on $2k+4$ vertices with exactly four hamiltonian cycles.}
\label{FILLINLABEL}
\end{figure}

\section*{Acknowledgements}
We thank Gunnar Brinkmann for valuable feedback concerning this manuscript, and in particular for his proof of Theorem~\ref{(3,2k)-regularUHG}. We are grateful to Jarne Renders for providing us with an independent algorithm to compute the number of hamiltonian cycles in a graph.

We gratefully acknowledge the support provided by the ORDinL project (FWO-SBO S007318N, Data Driven Logistics, 1/1/2018 - 31/12/2021). This research received funding from the Flemish Government under the ``Onderzoeksprogramma Artifici\"{e}le Intelligentie (AI) Vlaanderen'' programme.
The research of Jan Goedgebeur was supported by Internal Funds of KU Leuven.
The research of On-Hei Solomon Lo was supported by a Postdoctoral Fellowship of Japan Society for the Promotion of Science and by Natural Sciences and Engineering
Research Council of Canada.
The research of Ben Seamone was supported by Natural Sciences and Engineering
Research Council of Canada.
The research of Carol T.~Zamfirescu was supported by a Postdoctoral Fellowship of the Research Foundation - Flanders (FWO).

The computational resources and services used in this work were provided by the VSC (Flemish Supercomputer Center), funded by the Research Foundation - Flanders (FWO) and the Flemish Government – department EWI.

\newpage

\section*{Appendix}

\subsection*{Figures used in the proof of Theorem~\ref{variationTheorem}}


 	\begin{figure}[h!]
	\centering
  \includegraphics[width=0.8\linewidth]{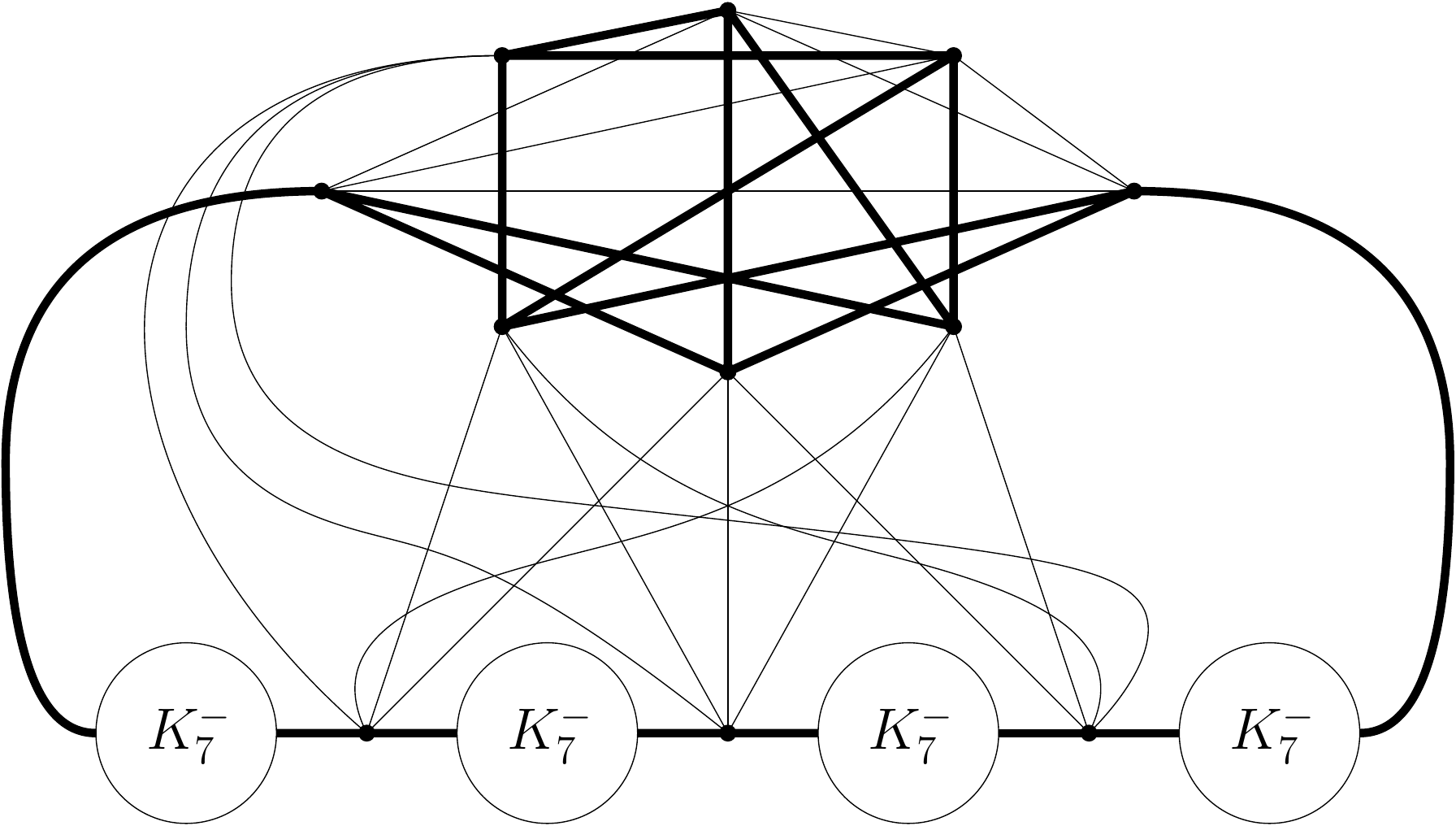}
\caption{A 6-regular graph on $39$ vertices with exactly $2^{13}\cdot 3^{4}\cdot 5^{4}$ hamiltonian cycles.}
\label{sixRegularVariation}       
\end{figure}
 	\begin{figure}[h!]
	\centering
  \includegraphics[width=0.8\linewidth]{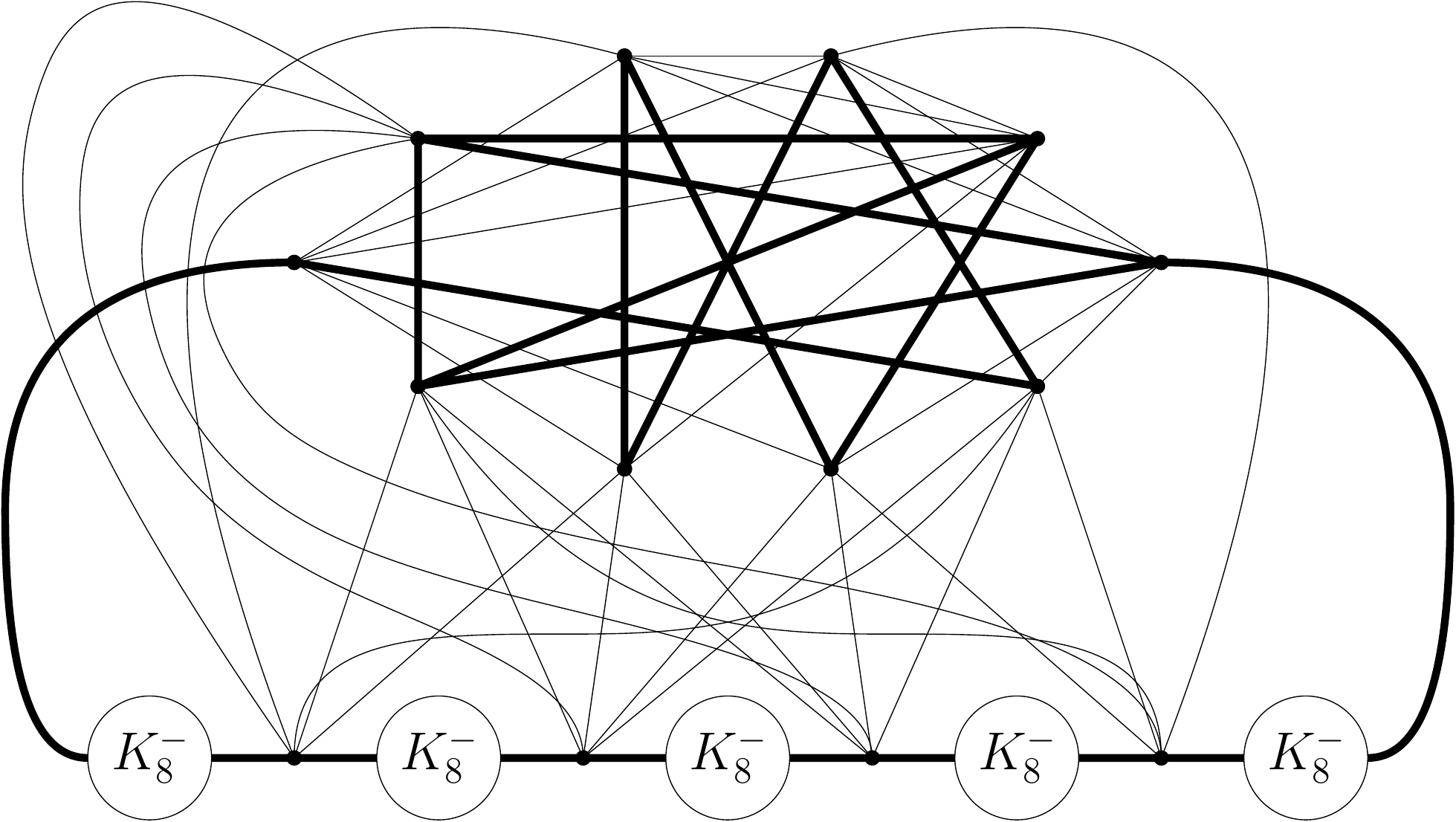}
\caption{A 7-regular graph on $54$ vertices with exactly $2^{21}\cdot3^{10} \cdot5^{5}$ hamiltonian cycles.}
\label{sevenRegularVariation}       
\end{figure}

\newpage

\section*{Supplementary material}

\subsection*{Graph counts related to Thomassen's independent dominating set method from Section~\ref{sect:rigd}}

\begin{table}[h!]
\centering
\begin{threeparttable}
	\begin{tabular}{c  r  r }
		\hline
		\noalign{\smallskip}
		Order & General & Bipartite\\
		\noalign{\smallskip}
		\hline
		\noalign{\smallskip}
		$\leq 8$ & $0$ & $0$\\
		$10$ & $110$ & $10$\\
		$12$ & $48\,577$ & $11\,180$\\
		\noalign{\smallskip}
		\hline
	\end{tabular}
	\end{threeparttable}	
\caption{Counts of all pairs $(G',\frak h)$ where $G'$ is a 3-regular graph and ${\frak h}$ is a hamiltonian cycle in its complement such that no minimal dominating set in $G'$ is ${\frak h}$-independent. The counts listed in the last column have the additional restriction that $G'$ is bipartite.
}
\label{table:counts_noRIGDSet_5reg}
\end{table}

\begin{table}[h!]
\centering
\begin{threeparttable}
	\begin{tabular}{c  r  r }
		\hline
		\noalign{\smallskip}
		Order & General & Bipartite\\
		\noalign{\smallskip}
		\hline
		\noalign{\smallskip}
		$\leq 13$ & $0$ & $0$\\
		$14$ & $7$ & $0$\\
		$15$ & $469$ & -\\		
		$16$ & ? & 3\\
		$17$ & ? & -\\
		$18$ & ? & 0\\
		$19$ & ? & -\\
		$20$ & ? & 0\\				
		\noalign{\smallskip}
		\hline
	\end{tabular}
	\end{threeparttable}	
\caption{Counts of all pairs $(G,\frak h)$ where $G$ is a 4-regular graph and ${\frak h}$ is a hamiltonian cycle in its complement such that no minimal dominating set in $G$ is ${\frak h}$-independent. The counts listed in the last column have the additional restriction that $G$ is bipartite.
}
\label{table:counts_noRIGDSet_6reg}
\end{table}

\newpage

\subsection*{Figures of the smallest uniquely hamiltonian \boldmath$(3,4)$-regular graphs from Observation~\ref{obs:UH_34}}

\begin{figure}[htb!]
	\centering
\begin{tikzpicture}[main_node/.style={circle,fill,draw,minimum size=0.5em,inner sep=.5pt]},scale=0.6]

\node[main_node] (0) at (0.8352681534044679, 0.048937304446440066) {};
\node[main_node] (1) at (1.6589333244843996, 1.235595757531027) {};
\node[main_node] (2) at (1.8098061849511238, 1.8861444018887408) {};
\node[main_node] (3) at (-0.9504792105886657, -0.20478094770215627) {};
\node[main_node] (4) at (-0.36517022872600124, -0.7620008317261995) {};
\node[main_node] (5) at (-2.2624241765999953, -1.0672871397830264) {};
\node[main_node] (6) at (-3.5802390642010575, -0.47816282396302423) {};
\node[main_node] (7) at (-4.269042636709932, 0.40558421749167195) {};
\node[main_node] (8) at (-2.719595680651345, 1.904918597092582) {};
\node[main_node] (9) at (-3.8824547460838694, 3.740943385186423) {};
\node[main_node] (10) at (-4.706196005566213, 2.554421509057696) {};
\node[main_node] (11) at (-4.857142857142858, 1.9037699559539742) {};
\node[main_node] (12) at (-1.9824749627221188, 3.851708285285943) {};
\node[main_node] (13) at (-2.682215080170673, 4.5518773790285145) {};
\node[main_node] (14) at (-0.7849541984835988, 4.857142857142858) {};
\node[main_node] (15) at (0.5329359427866187, 4.267998874395575) {};
\node[main_node] (16) at (1.2215766530949423, 3.3840926600126546) {};
\node[main_node] (17) at (-0.3278704492886977, 1.884715158021597) {};

 \path[draw, thick]
(0) edge node {} (1) 
(0) edge node {} (4) 
(0) edge node {} (17) 
(1) edge node {} (2) 
(1) edge node {} (14) 
(2) edge node {} (3) 
(2) edge node {} (16) 
(3) edge node {} (4) 
(3) edge node {} (5) 
(4) edge node {} (5) 
(4) edge node {} (15) 
(5) edge node {} (6) 
(5) edge node {} (10) 
(6) edge node {} (7) 
(6) edge node {} (13) 
(7) edge node {} (8) 
(7) edge node {} (11) 
(8) edge node {} (9) 
(8) edge node {} (17) 
(9) edge node {} (10) 
(9) edge node {} (13) 
(10) edge node {} (11) 
(11) edge node {} (12) 
(12) edge node {} (13) 
(12) edge node {} (14) 
(13) edge node {} (14) 
(14) edge node {} (15) 
(15) edge node {} (16) 
(16) edge node {} (17) 
;

\end{tikzpicture}	
\qquad
\begin{tikzpicture}[main_node/.style={circle,fill,draw,minimum size=0.5em,inner sep=.5pt]},scale=0.52]

\node[main_node] (0) at (1.5019798802589213, 0.6584893444606008) {};
\node[main_node] (1) at (1.1285251341438762, -0.4938881013124483) {};
\node[main_node] (2) at (0.2678619906848212, -1.2989934294425411) {};
\node[main_node] (3) at (-0.45882139191070026, 0.37411929438002156) {};
\node[main_node] (4) at (-2.23937620417457, -0.7169735489886673) {};
\node[main_node] (5) at (-2.2487068911402597, -1.3999016864867533) {};
\node[main_node] (6) at (-2.2793445722684407, -2.1237281972614186) {};
\node[main_node] (7) at (-4.857142857142858, -0.13739018491174448) {};
\node[main_node] (8) at (-3.149123771467701, 0.18820931147823972) {};
\node[main_node] (9) at (-4.546921085234617, 1.704372793047226) {};
\node[main_node] (10) at (-4.301554790769931, 2.9136609392019963) {};
\node[main_node] (11) at (-3.441023063271386, 3.718153460172287) {};
\node[main_node] (12) at (-2.885979638697187, 2.358221056965313) {};
\node[main_node] (13) at (-0.9339361793843719, 3.135719214425704) {};
\node[main_node] (14) at (-1.9105970110823254, 3.8190917672484925) {};
\node[main_node] (15) at (-1.737339988164189, 4.857142857142858) {};
\node[main_node] (16) at (1.0548871313902255, 2.985350398008154) {};
\node[main_node] (17) at (-2.2959217925581954, 1.54531334645144) {};

 \path[draw, thick]
(0) edge node {} (1) 
(0) edge node {} (5) 
(0) edge node {} (17) 
(1) edge node {} (2) 
(1) edge node {} (8) 
(2) edge node {} (3) 
(2) edge node {} (6) 
(3) edge node {} (4) 
(3) edge node {} (16) 
(4) edge node {} (5) 
(4) edge node {} (8) 
(5) edge node {} (6) 
(6) edge node {} (7) 
(7) edge node {} (8) 
(7) edge node {} (12) 
(8) edge node {} (9) 
(9) edge node {} (10) 
(9) edge node {} (14) 
(10) edge node {} (11) 
(10) edge node {} (17) 
(11) edge node {} (12) 
(11) edge node {} (15) 
(12) edge node {} (13) 
(13) edge node {} (14) 
(13) edge node {} (17) 
(14) edge node {} (15) 
(15) edge node {} (16) 
(16) edge node {} (17) 
;

\end{tikzpicture}
\qquad
\begin{tikzpicture}[main_node/.style={circle,fill,draw,minimum size=0.5em,inner sep=.5pt]},scale=0.53]

\node[main_node] (0) at (1.3489405832771286, 1.0178792916809163) {};
\node[main_node] (1) at (0.9443624764754297, -0.197455555764515) {};
\node[main_node] (2) at (0.07349518397716892, -0.9392761774255298) {};
\node[main_node] (3) at (-0.4879917643211389, 0.5936121946505857) {};
\node[main_node] (4) at (-2.420642767678236, -0.4083191670280293) {};
\node[main_node] (5) at (-1.769503148828262, -1.005968317995677) {};
\node[main_node] (6) at (-2.6029225344416043, -1.810224708519553) {};
\node[main_node] (7) at (-4.726419073131732, 0.4621964248914754) {};
\node[main_node] (8) at (-3.4685623110159765, 0.7659475024939151) {};
\node[main_node] (9) at (-4.857142857142858, 1.5630607388999334) {};
\node[main_node] (10) at (-4.771552429818314, 2.9388560351668724) {};
\node[main_node] (11) at (-3.996651742831611, 4.298489247290049) {};
\node[main_node] (12) at (-3.558791042000397, 3.7385592941860786) {};
\node[main_node] (13) at (-2.040822021521665, 3.4250988219056246) {};
\node[main_node] (14) at (-2.2967519264893332, 4.031227196902193) {};
\node[main_node] (15) at (-0.47540753553456705, 4.857142857142858) {};
\node[main_node] (16) at (1.3392345390307072, 3.2355681212448153) {};
\node[main_node] (17) at (0.10637410520511636, 2.8465767056775926) {};

 \path[draw, thick]
(0) edge node {} (1) 
(0) edge node {} (5) 
(0) edge node {} (17) 
(1) edge node {} (2) 
(1) edge node {} (8) 
(2) edge node {} (3) 
(2) edge node {} (6) 
(3) edge node {} (4) 
(3) edge node {} (16) 
(4) edge node {} (5) 
(4) edge node {} (8) 
(5) edge node {} (6) 
(6) edge node {} (7) 
(7) edge node {} (8) 
(7) edge node {} (9) 
(7) edge node {} (12) 
(8) edge node {} (9) 
(9) edge node {} (10) 
(9) edge node {} (14) 
(10) edge node {} (11) 
(10) edge node {} (17) 
(11) edge node {} (12) 
(11) edge node {} (15) 
(12) edge node {} (13) 
(13) edge node {} (14) 
(13) edge node {} (17) 
(14) edge node {} (15) 
(15) edge node {} (16) 
(16) edge node {} (17) 
;

\end{tikzpicture}

\bigskip

\begin{tikzpicture}[main_node/.style={circle,fill,draw,minimum size=0.5em,inner sep=.5pt]},scale=0.53]

\node[main_node] (0) at (1.8095291491165888, 1.5470449254096619) {};
\node[main_node] (1) at (1.56019225702003, 0.4084370250407998) {};
\node[main_node] (2) at (0.786794000818718, -0.8460755943832616) {};
\node[main_node] (3) at (-0.0022062514028835523, -1.4370628446376523) {};
\node[main_node] (4) at (-1.138123259158874, -1.6523945931988404) {};
\node[main_node] (5) at (-2.024160065102416, -1.7657015524512756) {};
\node[main_node] (6) at (-2.9370488010901132, -1.323459307808765) {};
\node[main_node] (7) at (-2.9678502983305197, 0.8358870446565665) {};
\node[main_node] (8) at (-3.738062562556978, 0.12231632185737595) {};
\node[main_node] (9) at (-4.857142857142858, 1.5440291303851004) {};
\node[main_node] (10) at (-4.6077666088555365, 2.682790612375079) {};
\node[main_node] (11) at (-3.834345034059446, 3.937413748808215) {};
\node[main_node] (12) at (-3.0738339414507987, 4.499412647058845) {};
\node[main_node] (13) at (-1.9522560921427594, 4.700775132533947) {};
\node[main_node] (14) at (-1.023462057333651, 4.857142857142858) {};
\node[main_node] (15) at (-0.11057095975559994, 4.41456712023003) {};
\node[main_node] (16) at (-0.07956226795593446, 2.2554861976935907) {};
\node[main_node] (17) at (0.6905492811291136, 2.9689070362640386) {};

 \path[draw, thick]
(0) edge node {} (1) 
(0) edge node {} (5) 
(0) edge node {} (17) 
(1) edge node {} (2) 
(1) edge node {} (8) 
(2) edge node {} (3) 
(2) edge node {} (6) 
(3) edge node {} (4) 
(3) edge node {} (16) 
(4) edge node {} (5) 
(4) edge node {} (8) 
(5) edge node {} (6) 
(6) edge node {} (7) 
(7) edge node {} (8) 
(7) edge node {} (12) 
(7) edge node {} (16) 
(8) edge node {} (9) 
(9) edge node {} (10) 
(9) edge node {} (14) 
(10) edge node {} (11) 
(10) edge node {} (17) 
(11) edge node {} (12) 
(11) edge node {} (15) 
(12) edge node {} (13) 
(13) edge node {} (14) 
(13) edge node {} (17) 
(14) edge node {} (15) 
(15) edge node {} (16) 
(16) edge node {} (17) 
;

\end{tikzpicture}
\qquad
\begin{tikzpicture}[main_node/.style={circle,fill,draw,minimum size=0.5em,inner sep=.5pt]},scale=0.57]

\node[main_node] (0) at (1.8094329018427038, 1.9061458956048358) {};
\node[main_node] (1) at (1.2924545632668636, 0.034537238702179174) {};
\node[main_node] (2) at (0.7584297013301446, 1.0472808327236982) {};
\node[main_node] (3) at (0.3903154758049556, -0.5335758919016733) {};
\node[main_node] (4) at (-1.0662591259658143, -1.2746797702046955) {};
\node[main_node] (5) at (-2.361701489381489, -1.4295347117768988) {};
\node[main_node] (6) at (-3.263285910900071, -0.841394396930129) {};
\node[main_node] (7) at (-4.062310475438718, -0.22829761841482288) {};
\node[main_node] (8) at (-4.602494548299319, 0.6412684544586176) {};
\node[main_node] (9) at (-4.857142857142858, 1.5231535916121617) {};
\node[main_node] (10) at (-4.33854492615805, 3.3929116242675086) {};
\node[main_node] (11) at (-3.788370584588046, 2.505910075387783) {};
\node[main_node] (12) at (-3.436638598411978, 4.017301449530823) {};
\node[main_node] (13) at (-1.9811419663804273, 4.7019469690767615) {};
\node[main_node] (14) at (-0.6857025014859501, 4.857142857142858) {};
\node[main_node] (15) at (0.1881831561001608, 4.355287962901201) {};
\node[main_node] (16) at (1.0161551041513661, 3.6135937355473153) {};
\node[main_node] (17) at (1.6372837853314035, 2.957120797517044) {};

 \path[draw, thick]
(0) edge node {} (1) 
(0) edge node {} (14) 
(0) edge node {} (17) 
(1) edge node {} (2) 
(1) edge node {} (3) 
(2) edge node {} (3) 
(2) edge node {} (13) 
(2) edge node {} (16) 
(3) edge node {} (4) 
(3) edge node {} (15) 
(4) edge node {} (5) 
(4) edge node {} (11) 
(5) edge node {} (6) 
(5) edge node {} (9) 
(6) edge node {} (7) 
(6) edge node {} (12) 
(7) edge node {} (8) 
(7) edge node {} (11) 
(8) edge node {} (9) 
(8) edge node {} (17) 
(9) edge node {} (10) 
(10) edge node {} (11) 
(10) edge node {} (12) 
(11) edge node {} (12) 
(12) edge node {} (13) 
(13) edge node {} (14) 
(14) edge node {} (15) 
(15) edge node {} (16) 
(16) edge node {} (17) 
;

\end{tikzpicture}

\caption{The smallest uniquely hamiltonian $(3,4)$-regular graphs.}

\label{fig:UH_34}
\end{figure}
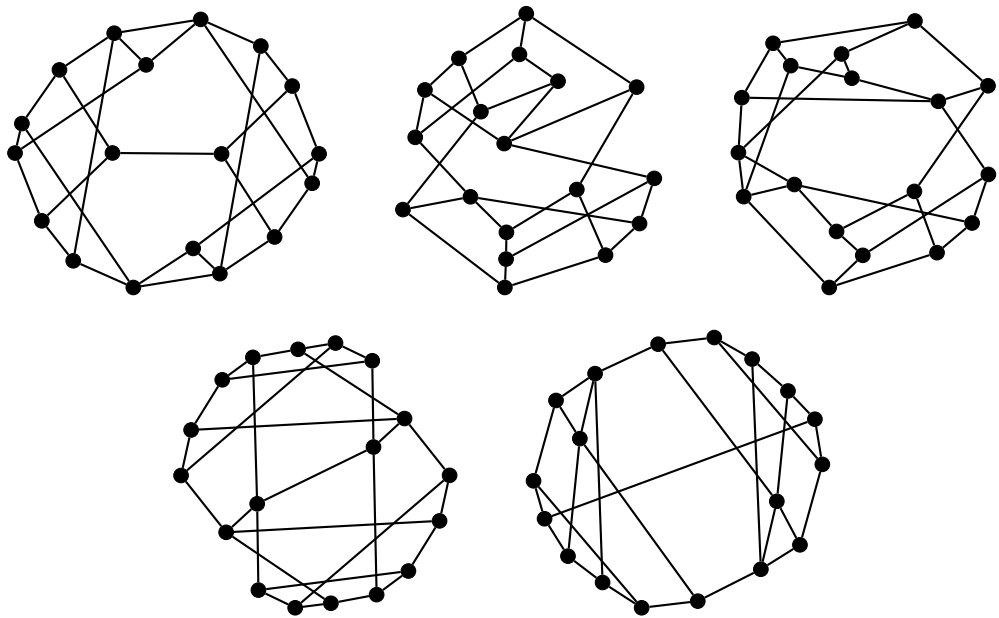

\end{document}